%% file: ASEP-arxiv3.tex
\documentclass[11pt]{article}

\usepackage{color}    
\usepackage{epsfig} 
\usepackage{epsf}            
\usepackage{float}           

\usepackage[margin=1cm]{caption}

\usepackage{amsthm}
\usepackage{graphics}
\usepackage{amsfonts}
\usepackage{amsxtra}
\usepackage{amsmath}
\usepackage{a4}
\usepackage{amssymb}
\newtheorem{theorem}{Theorem}
\newtheorem{lemma}[theorem]{Lemma}
\newtheorem{corollary}[theorem]{Corollary}

\newtheorem{proposition}[theorem]{Proposition}

\newtheorem{algorithm}{Algorithm}
\numberwithin{theorem}{section}
\numberwithin{equation}{section}
\numberwithin{figure}{section}

\include{jbmmacros}

\newcommand{\malpha}{{{\alpha}}}
\newcommand{\mepsilon}{{{\epsilon}}}
\newcommand{\mdelta}{{{\delta}}}
\newcommand{\mI}{{I}}

\newcommand{\tQ}{\widetilde{Q}}

\newcommand{\trace}{\operatorname{trace}}

\newcommand{\ind}{\boldsymbol{1}}

\begin{document}

\title{Stationary distributions of the multi-type
ASEP}

\author{
\textbf{James B.~Martin}
\\
\textit{University of Oxford}
\\
\texttt{martin@stats.ox.ac.uk}
}
\date{}
\maketitle

\begin{abstract}
We give a recursive construction 
of the stationary distribution of multi-type
asymmetric simple exclusion processes on a finite ring  
or on the infinite line $\ZZ$. The construction can 
be interpreted in terms of ``multi-line diagrams" or 
systems of queues in tandem. Let $q$ be the asymmetry parameter
of the system. The queueing construction generalises the one
previously known for the totally asymmetric ($q=0$) case, by introducing
queues in which each potential service is unused with probability $q^k$ when the queue-length is $k$. The analysis is based on the
matrix product representation of Prolhac, Evans and Mallick. 
Consequences of the construction include: a simple 
method for sampling exactly
from the stationary distribution for the system on a ring; results on common denominators of
the stationary probabilities, expressed as rational functions
of $q$ with non-negative integer coefficients;
and probabilistic descriptions of ``convoy formation" phenomena 
in large systems. 
\end{abstract}

\section{Introduction and main results}\label{sec:intromain}
The \textit{asymmetric simple exclusion process}
(ASEP) is one of the simplest examples of an interacting
particle system. In physcial terms 
it can be seen as a basic model of non-reversible flow. 
Despite its simplicity, it displays an extremely 
rich behaviour, and attracts intense study in physics,
probability, combinatorics, and beyond. 

In this article we consider ASEPs with multiple types 
of particle. We start by defining the model on the ring 
$\Z_L=\{0,1,2,\dots, L-1\}$ (with cyclic bound\-ary conditions).
The system is a continuous-time Markov chain with state-space $\{1,2,\dots, N,\infty\}^{\ZZ_L}$, for some $N\geq 1$. For a configuration $\eta=(\eta_i, i\in\ZZ_L)$, we say that $\eta_i$ is the type of the particle at site $i$. (We may sometimes refer to particles of type 
$\infty$ as \textit{holes} -- we could equally give them label
$N+1$ rather than $\infty$, but the different notation
is sometimes helpful.)

Fix some $q\geq0$.
The dynamics of the process are as follows. 
If
$\eta(i)>\eta(i+1)$ (where the addition $i+1$ is done mod $L$), then the values $\eta(i)$ and $\eta(i+1)$ 
are exchanged at rate $1$. If instead $\eta(i)<\eta(i+1)$, 
then the values $\eta(i)$ and $\eta(i+1)$ are exchanged
at rate $q$. That is, neighbouring particles rearrange
themselves into increasing order at rate $1$, and 
into decreasing order at rate $q$. 
We will assume $q\in[0,1)$ throughout. 
Such a process was perhaps first described by
Alcaraz and Rittenberg \cite{AlcarazRittenberg} 
as an example of a more general class of multi-type particle systems. 

The dynamics preserve the number of particles of each type. 
Restricted to a state space with a given number of particles 
$k_n$ of each type, for $n=1,\dots,N$ and $k_1+k_2+\dots+k_N\leq L$, the process is irreducible. Our particular focus is 
on the stationary distribution. 

Consider first the case $N=1$. This is the standard (one-type) ASEP
(for an extensive introduction, see Part III of the book of Liggett
\cite{Liggettbook2}). 
For the one-type ASEP on the ring, each site
of $\ZZ_L$ contains either a particle (type $1$) or a hole (type $\infty$). A particle exchanges places with a hole to its left at rate $1$, and with a hole to its right at rate $q$.  Fix the number of particles to be $k$. Then it's well known
(and straightforward to show) that the stationary distribution 
of the process is uniform on all $\ch{L}{k}$ configurations. 

The ASEP with types $\{1,\dots, N, \infty\}$ can be viewed as a coupling of $N$ one-type ASEPs. Specifically, for any $n=1,\dots, N$, we can consider a projection under which types $r\leq n$ are considered ``particles" and types $r>n$ are considered ``holes"; for each $n$, the image of the process
under this projection is a one-type ASEP. In particular, 
projecting the stationary distribution in this way gives
a uniform distribution on the corresponding state-space.

However, despite the fact that all these projections are uniform, it is far from the case that the full stationary distribution is uniform. See for example Figure 
\ref{fig:convoys} for samples from the stationary distribution 
of a system with $1000$ particles all of different types, on the ring with $1000$ sites, for different values of $q$. A striking feature of the configurations observed is the appearance of long strings of particles with similar labels
(``convoys"). This ``clustering" is more pronounced for small $q$, and more pronounced around the middle of the range of 
particle types. 

\begin{figure}
\begin{center}

\includegraphics[trim=0 2cm 0 2cm, clip, width=0.97\textwidth]{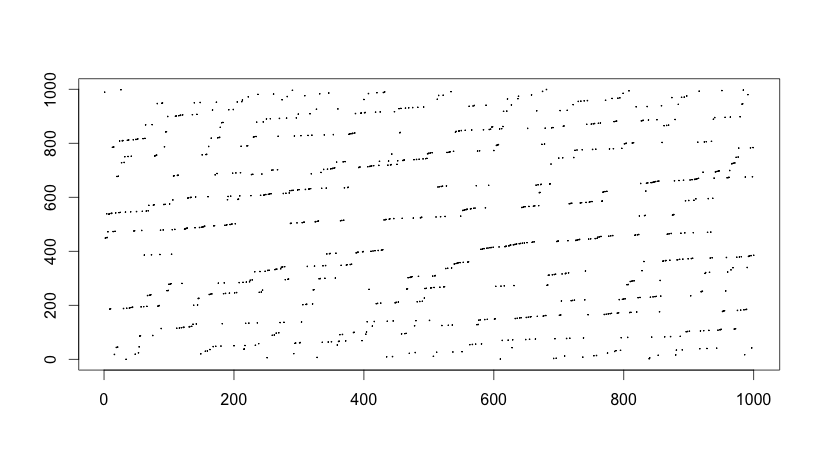}
\vspace{-0.1cm}

\includegraphics[trim=0 2cm 0 2cm, clip, width=0.97\textwidth]{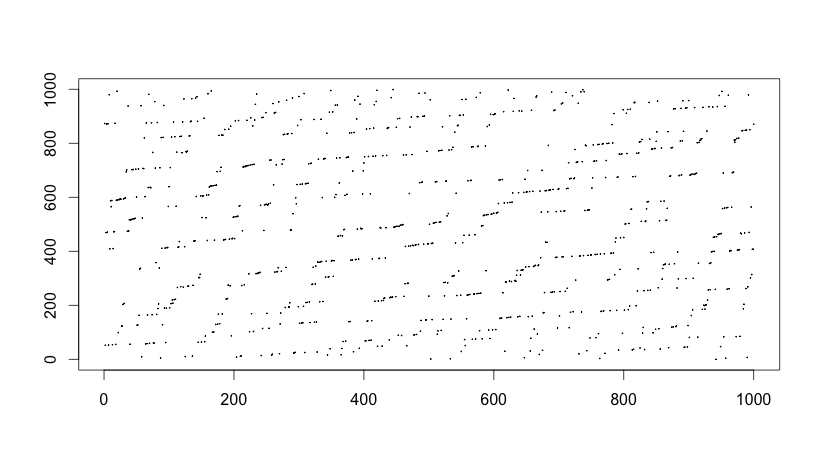}
\vspace{-0.1cm}

\includegraphics[trim=0 2cm 0 2cm, clip, width=0.97\textwidth]{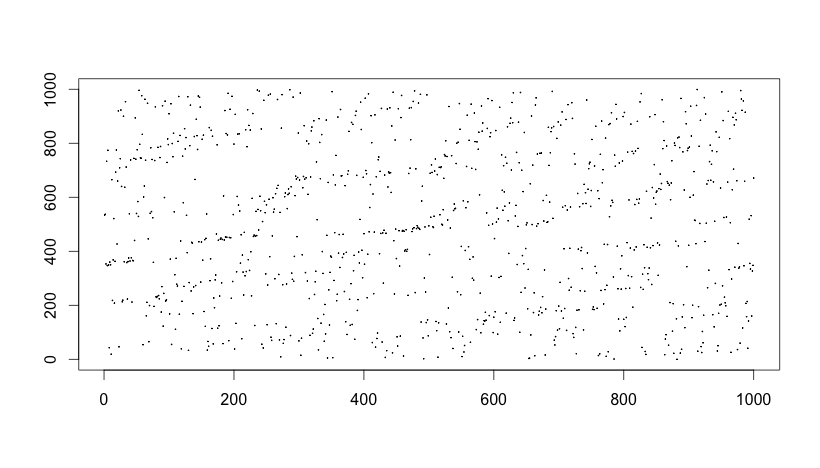}
\caption{\label{fig:convoys}
Samples of the stationary distribution 
of an ASEP on a ring of $1000$ sites, with $1000$ different
types of particle. 
The horizontal axis is labelled by site,
and the vertical axis by particle type.
Top: $q=0$; middle: $q=0.1$;
bottom: $q=0.8$. The ``convoys" are more pronounced
when $q$ is small (and the effect is stronger
around the middle of the range of particle classes). 
See the results in Section \ref{sec:convoys}.}
\end{center}
\end{figure}

Before stating the main results of the paper, we 
mention certain previous results --
see Section \ref{sec:related} for further background.
The stationary distribution of the 2-type ASEP 
(where the non-uniformity mentioned in the previous paragraph 
already appears) was 
studied by Derrida, Janowksy, Lebowitz and Speer \cite{DJLStasep}
(see also \cite{DEHP}) using the matrix product ansatz.
Further combinatorial descriptions in the case of the TASEP
(\textit{totally asymmetric}, i.e.\ $q=0$) were given by Ferrari, Fontes and Kohayakawa 
\cite{FFKtasep}, Angel \cite{Angeltasep}, and Duchi and Schaeffer
\cite{DucSch}. In \cite{FerMarmulti}, Ferrari and Martin introduced
a new recursive method which gives a construction of the stationary 
distribution of the $N$-type TASEP for any $N$. 

The construction by Ferrari and Martin is written in terms
of a system of queues in tandem (or so-called ``multiline diagrams"),
and the proof uses time-reversal arguments. 
A different approach, but exploiting a related recursive 
structure, was 
then found by Evans, Ferrari and Mallick \cite{EvansFerrariMallick}, to give a 
representation of the stationary distribution of the 
$N$-type TASEP using the matrix product ansatz 
(this was elaborated further for example by Arita, Ayyer, 
Mallick and Prolhac in \cite{AAMPtasep}). 

This matrix-product representation was then extended
to $q>0$ by 
Prolhac, Evans and Mallick in \cite{PEM}. They raised the question 
of whether the more proba\-bilistic/combinatorial 
construction of \cite{FerMarmulti} could also be generalised
to the ASEP. In this paper we show that indeed it can, 
giving a multi-line queue construction of the stationary
distribution of the ASEP. 

This construction has a number of nice consequences:
\begin{itemize}
\item
algorithmic: it gives an efficient method for sampling exactly from
systems even with a large number of types (see Algorithm 
\ref{multialgo} below, which was used to generate the samples in 
Figure \ref{fig:convoys} with $1000$ types of particle);
\item
algebraic: it leads directly to an expression 
for a common denominator of the stationary probabilities 
of an ASEP system, expressed as rational functions of $q$
(see Theorem \ref{thm:denominators});
\item
probabilistic: it can be exploited to give 
interesting qualitative and quantitative
information about the properties of the stationary distribution,
including the clustering properties described in 
Theorem \ref{thm:Wlimit}. 
\end{itemize}

A central object in the construction for the TASEP in \cite{FerMarmulti} is a simple model of a priority queue in discrete time, with a Markovian service process. When a service occurs in the
queue, the customer in the queue with highest priority departs
(unlesss the queue is empty, in which case the service is unused).
To generalise to $q>0$, we introduce a queue with ``rejected 
services". When a service occurs, it is offered to 
each customer in turn, in order of priority. Each customer
accepts the service with probability $1-q$; the first customer to 
accept it departs (unless they all reject it, in which case the service is unused). The multi-line process can be seen
as a collection of several such queues in series; such a model
is closely related to the so-called \textit{$q$-TASEP} (especially the discrete-time
version considered for example in \cite{BorodinCorwinqTASEP}).

It would be possible to give a proof of our main result
by an elaboration of the methods involving dynamic reversibility which 
were used in \cite{FerMarmulti} \cite{FerMarHAD}.
Instead, in this paper we rely 
directly on the matrix-product representation of Prolhac, Evans and Mallick, which we explain in Section \ref{sec:matrices}.
We make use of particular instances of matrices
which satisfy the quadratic algebra of \cite{PEM}; these can be related
to Markov transition matrices which govern the evolution 
of a queue. 

\subsection{Algorithms and multi-line diagrams}
\label{sec:main-algorithms}
We now begin to state the main results of the paper. 
We start with an algorithm
for sampling from the stationary distribution
of the $N$-type ASEP on $\ZZ_L$, with given
particle counts. 
Later in the paper we will be able to restate
it in terms of a function which assigns
weights to ``multi-line diagrams", such
that the probability distribution on the diagrams
proportional to these weights projects to the ASEP
stationary distribution.

We describe the algorithm using queueing language, 
although we postpone until later the introduction
of notation which describes queue-length processes more
formally (see Section \ref{sec:queueing}). The sites of $\ZZ_L$ are
treated as ``times" in the queueing process. Hence time is cyclic
(a setting in which time is more straightforwardly indexed by $\ZZ$ is 
introduced in Section \ref{sec:queueing} where we consider
ASEPs on $\ZZ$ rather than $\ZZ_L$). 

First consider the case $N=2$. Fix $L$, $k_1$, $k_2$ with
$k_1+k_2\leq L$. The following algorithm can be used to 
generate a sample from the stationary distribution of the $2$-type 
ASEP on $\ZZ_L$ with $k_1$ type-$1$ particles and $k_2$ type-$2$ particles. The algorithm does the following thing:
(i) it randomly chooses a set of $k_1$ arrival times, and a set of $k_1+k_2$ service times; (ii) it assigns to each arrival time
a different departure time chosen from among the service times. 
This gives an output of the queue, in which each time $i\in\ZZ_L$ is 
either a departure time, an unused service time, or a time which had no service. This is used to define a configuration in $\{1,2,\infty\}^{\ZZ_L}$.

\begin{algorithm}\label{2algo}
$\, $
\begin{itemize}
\item Choose a set of ``arrival times" uniformly among all subsets of $\ZZ_L$
of size $k_1$, and independently a set of ``service times" uniformly among all subsets of $\ZZ_L$ of size $k_1+k_2$. 
\item
Now process
the arrivals one by one, in an arbitrary order (for example left to right), and assign a departure time to each 
one from among the service times, in the following way.

Suppose we have already processed $r$ of the $k_1$ arrivals,
  where $0\leq r<k_1$.  This means we have already assigned $r$ of the
  $k_1+k_2$ service times.  Now look at an arrival we have not
  processed yet -- say it occurs at time $i$.  
We wish to assign a service time to be the departure time of the arrival at time $i$, 
which has not yet been assigned to any other arrival.  
  If there is a service
  at time $i$ which has not yet been assigned, then assign it to this
  arrival, and we are done. Otherwise, let the remaining $k_1+k_2-r$
  available potential service times be $i_1, i_2,\dots, i_{k_1+k_2-r}$; we list
  these sites in cyclic order around the ring starting from $i$, so
  that
\[
0<[(i_1-i)\bmod L] < [(i_2-i)\bmod L] < \dots < [(i_{k_1+k_2-r}-i)\bmod L].
\]
Now assign the arrival at $i$ to the service at $i_j$ 
with probability 
\begin{equation}\label{truncate}
q^{j-1}/(1+q+q^2+\dots+q^{k_1+k_2-r-1}).
\end{equation} 
\item
Having done this for all $k_1$ arrivals, we have chosen $k_1$ departure times, and the remaining
$k_2-k_1$ service times will be unused. Now define an output configuration $D\in\{1,2,\infty\}^{\ZZ_L}$ as follows: for each $i\in\ZZ_L$,
\begin{equation}\label{2statring}
D_i=\begin{cases}
1&\text{ if a departure occurs at } i\\
2&\text{ if an unused service occurs at } i\\
\infty&\text{ if no service was available at } i\\
\end{cases}.
\end{equation}

\end{itemize}
\end{algorithm}

\begin{theorem}\label{thm:2algo}
Algorithm \ref{2algo} generates a configuration $D_i, i\in\ZZ_L$ whose distribution is the stationary distribution of the ASEP
on $\ZZ_L$ with $k_1$ type-$1$ particles and
$k_2$ type-$2$ particles. 
\end{theorem}

Now we generalise this algorithm to all $N\geq 2$, 
in a recursive way.  
Fix $L$ and $k_1, \dots, k_N$. The following algorithm
generates an algorithm from the stationary distribution
of the $N$-type ASEP on $\ZZ_L$ with particle counts $k_1,
\dots, k_N$ (that is, with $k_n$ particles of type $n$ for
$1\leq n\leq N$). The input into the algorithm
is a collection of $k_1+\dots+k_N$ service times, along with
an arrival process $A=(A_i, i\in\ZZ_L)\in\{1,2,\dots, N,\infty\}^{\ZZ_L}$, which is itself drawn from the stationary distribution 
of the $(N-1)$-type ASEP on $\ZZ_L$ with particle counts
$k_1,\dots, k_{N-1}$. If $A_i=n$ where $1\leq n\leq N$,
there is an arrival of type $n$ at time $i$. If $A_i=\infty$,
there is no arrival at time $i$. The algorithm 
gives an output in which each time in $\ZZ_L$ is a departure 
type of some customer with a type $n\in\{1,2,\dots, N-1\}$, 
or an unused service time, or a time which had no service. 
This is used to define a configuration in $\{1,2,\dots, N\}^{\ZZ_L}$.

\begin{algorithm}\label{multialgo}
$\, $
\begin{itemize}
\item Choose the arrival process according to the 
stationary distribution of the $(N-1)$-type ASEP
on $\ZZ_L$ with 
particle counts
$k_1,k_2,\dots, k_{N-1}$, and independently
choose the set of service times uniformly among all subsets of
size $K=k_1+k_2+\dots+k_N$. 
\item
Now consider the arrival times one by one, 
as in Algorithm \ref{2algo}, but importantly this is
now done in order of type. We start by considering 
each of the type-1 arrivals, then each of the type-2 arrivals, and so on. The order in which arrivals of the same type are considered is arbitrary (for example, left to right). 

We assign a different service time to each considered arrival in turn, as follows. 
Suppose we have already processed $r$ of the arrivals,
where $0\leq r<k_1+k_2+\dots+k_{N-1}$.  This means we have already assigned 
$r$ of the $K$ service times.
Now look at an arrival we have not processed yet -- say it occurs at time $i$. We wish to assign it a departure time from among
those service times which are not yet assigned.
If there is a service
at time $i$ which has not yet been assigned, 
then assign it to this arrival, and we are done. Otherwise,
let the remaining $K-r$ available service times be $i_1, i_2,\dots, i_{K-r}$; we list
these sites in cyclic order around the ring starting from $i$, so that
\[
0<[(i_1-i)\!\!\mod L] < [(i_2-i)\!\!\mod L] < \dots < [(i_{K-r}-i)\!\!\mod L].
\]
Now we assign the arrival at $i$ to the service at $i_j$ 
with probability 
\begin{equation}\label{truncate2}
q^{j-1}/(1+q+q^2+\dots+q^{K-r-1}).
\end{equation} 
\item
Having done this for all arrivals, we have chosen $k_1+k_2+\dots+k_{N-1}$ 
departure times, and the remaining
$k_{N}$ service times are unused. Now define the configuration 
$D\in\{1,2,\dots, N\}^{\ZZ_L}$ as follows: 
\begin{equation}\label{multistatring}
D_i=\begin{cases}
n&\text{ if a departure of type $n$ arrival occurs at } i, \text{ for }1\leq n\leq N-1\\
N&\text{ if an unused service occurs at } i\\
\infty&\text{ if no service was available at } i\\
\end{cases}.
\end{equation}

\end{itemize}
\end{algorithm}

\begin{theorem}\label{thm:multialgo}
Algorithm \ref{multialgo} generates a configuration $(D_i, i\in\ZZ_L)$ whose distribution is the stationary distribution of the $N$-type ASEP on $\ZZ_L$ with particle counts $k_1,\dots, k_N$.
\end{theorem}

To generate a sample from an $N$-type TASEP, we need to 
apply Algorithm \ref{multialgo} $N-1$ times in all. 
The $n$th iteration takes as arrival process a configuration 
whose distribution is stationary for an $n$-type ASEP, 
along with an independent service process, and outputs a configuration whose distribution is stationary for an $(n+1)$-type
ASEP. (The first iteration, with $n=1$, is equivalent to Algorithm \ref{2algo}.)

We can combine the $N-1$ iterations into a 
``multi-line diagram" with $N$ lines, 
as was done for the $q=0$ case in \cite{FerMarmulti}.
The $n$th line of
the diagram is a $n$-type ASEP configuration, i.e.\ 
a configuration in $\{1,\dots,n,\infty\}^{\ZZ_L}$, 
with particle counts $k_1, \dots, k_n$. 
It is the arrival process for the $n$th iteration
of the algorithm (if $1\leq n\leq N-1$) and the
output configuration of the $(n-1)$st iteration of the algorithm
(if $2\leq n\leq N$). Specifically, the last line of
the diagram gives a sample from the $N$-type stationary distribution.

\begin{figure}[htb]
\begin{center}
\boxed{
\includegraphics[width=0.75\textwidth]{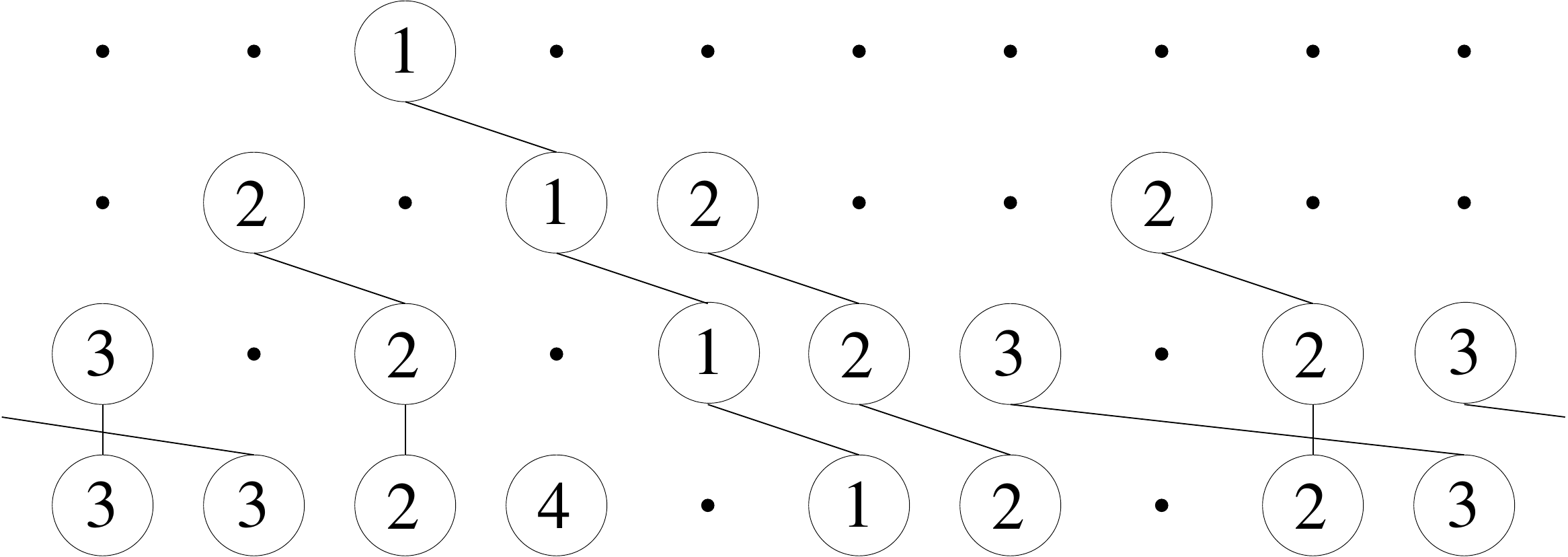}
}
\vspace{0.5cm}

\boxed{
\includegraphics[width=0.75\textwidth]{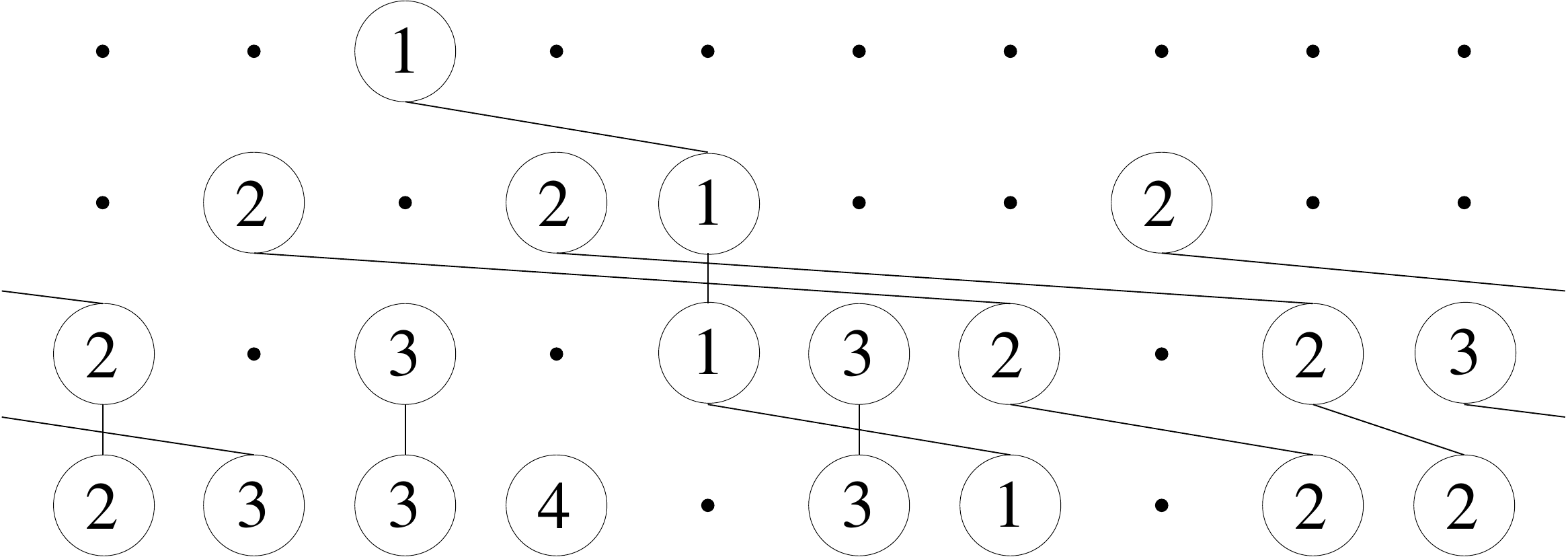}
}
\caption{\label{fig:multiline}
Examples of multi-line diagrams, 
for $N=4$, $L=10$, and $(k_1, k_2, k_3, k_4)=(1,3,3,1)$.
The configuration of occupied sites is the same in both cases.
The upper diagram then shows the deterministic assignment of types to 
particles in the case $q=0$. The lower diagram shows one of the 
(many) other possible assignments of types to particles when $q>0$,
and ``rejected services" are allowed.}
\end{center}
\end{figure}

On the $n$th line, the number of occupied sites (i.e.\ the sites where the value of the configuration is finite rather than infinity) is $k_1+\dots+k_n$. Ignoring the types, the set of occupied sites
is uniform among all subsets of $\ZZ_L$ of size $k_1+\dots+k_n$;
furthermore, the sets of occupied sets on different lines is 
independent. 

Note that for $q=0$ (the case of the TASEP), the only randomness
in Algorithm \ref{multialgo} is in the first step when the arrival and service processes are chosen. In that case, whenever the algorithm looks to assign a departure time to an arrival at time $i$, it chooses the first so-far unassigned service time in $i, i+1, i+2, \dots$. In this case, given the sets of occupied sites on 
the different lines of the multi-line diagram, the assignment of 
types to the occupied sites is deterministic. 
This gives the original algorithm of \cite{FerMarmulti}. 

Figure \ref{fig:multiline} shows two multi-line diagrams
for the case $N=4$, $L=10$, and $(k_1, k_2, k_3, k_4)=(1,3,3,1)$.
The sets of occupied sites are the same for the two 
diagrams. The upper diagram is the one which results (deterministically) from the configuration of occupied sites in the case $q=0$. The lower diagram shows another possible configuration 
(one of many) when $q>0$. 

\subsection{Common denominators}
\label{sec:main-denominators}
By considering the various sources of randomness
that go into the construction of a multi-line diagram
using repeated applications of Algorithm \ref{multialgo},
we can obtain the following result
giving a common denominator for the
ASEP probabilities. We write 
$[n]_q=1+q+\dots+q^{n-1}$, and 
$[n]_q!=[1]_q[2]_q\dots[n]_q$.
\begin{theorem}\label{thm:denominators}
Consider the $N$-type ASEP on $\ZZ_L$,
with particle counts $(k_1, \dots, k_N)$. 
The stationary probability of every state
is a polynomial in $q$ with non-negative integer
coefficients divided by the common denominator
\begin{multline}
\label{generaldenom}
\Ch{L}{k_1}
\Ch{L}{k_1+k_2}
\dots
\Ch{L}{k_1+k_2+\dots+k_N}
\\
\times
\frac{
[k_1+k_{2}]_q!
}
{
[k_{2}]_q!
}
\frac{
[k_1+k_2+k_{3}]_q!
}
{
[k_{3}]_q!
}
\dots
\frac{
[k_1+k_2+\dots+k_{N}]_q!
}
{
[k_{N}]_q!
}.
\end{multline}
\end{theorem}


A special case is that of a process on $\ZZ_L$
with $L$ different classes of particle, as depicted in Figure \ref{fig:convoys}.
All other processes
on $\ZZ_L$ can be obtained as projections of this process.
We can take $N=L-1$ and $k_1=k_2=\dots=k_{L-1}=1$
(it is equivalent to regard the final particle as a
particle of type $L$ or as a hole). This gives the following result:

\begin{corollary}\label{cor:Ldenominators}
The stationary probabilities of the ASEP on $\ZZ_L$
with one particle of each of $L$ different classes
are given by polynomials in $q$ with non-negative
integer coordinates divided by the common denominator
\begin{equation}\label{predenom}
\Ch{L}{1}
\Ch{L}{2}
\dots
\Ch{L}{L-1}
[2]_q!
[3]_q!
\dots
[L-1]_q!.
\end{equation}
\end{corollary}
The expression in (\ref{predenom}) can also 
be written as
\begin{equation}
\label{Ltypedenom}
\Ch{L}{1}
\Ch{L}{2}
\dots
\Ch{L}{L-1}
(1+q)^{L-2}(1+q+q^2)^{L-3}\dots(1+q+q^2+\dots+q^{L-2})^1.
\end{equation}
We note that a related formula appears in work of Cantini, de Gier and Wheeler \cite{CdGW} in a more general setting of Macdonald polynomials; specialising their formula in their Section 5 to the ASEP, one obtains an expression like (\ref{Ltypedenom}) but with an extra factor of $(1-q)^{L(L-1)/2}$, and without the restriction to non-negative integer coefficients in the numerator. 

The denominators given by (\ref{generaldenom}) 
and (\ref{Ltypedenom}) may well not be the best possible. 
In fact, these are common denominators for the probabilities
of all multi-line diagrams; the projection from 
the multi-line diagrams to the bottom line giving
a single ASEP configuration may introduce further common
factors. 

\begin{figure}[htb]
\begin{minipage}{0.325\linewidth}
\centering
\includegraphics[width=0.6\linewidth]{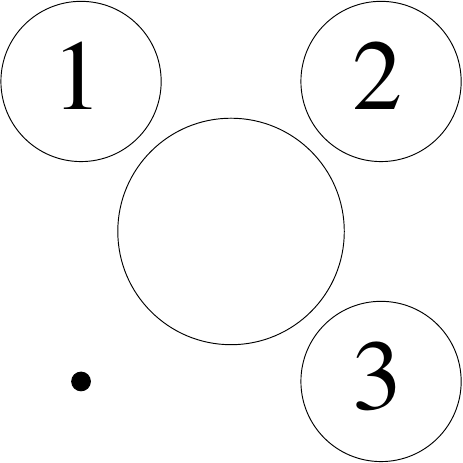}
\\

\vspace{0.5cm}
$\displaystyle
\frac{9+7q+7q^2+q^3}{96(1+q)(1+q+q^2)}$
\end{minipage}
\begin{minipage}{0.325\linewidth}
\centering
\includegraphics[width=0.6\linewidth]{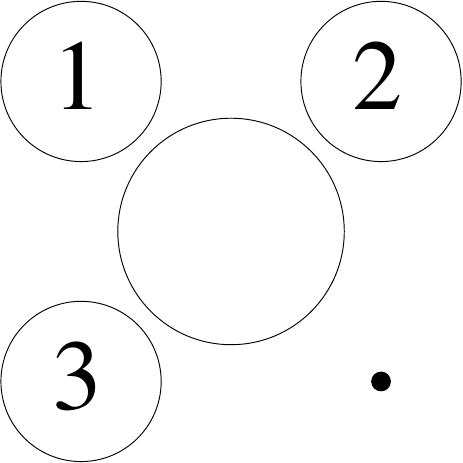}
\\

\vspace{0.5cm}
$\displaystyle
\frac{3+9q+9q^2+3q^3}{96(1+q)(1+q+q^2)}$
\end{minipage}
\begin{minipage}{0.325\linewidth}
\centering
\includegraphics[width=0.6\linewidth]{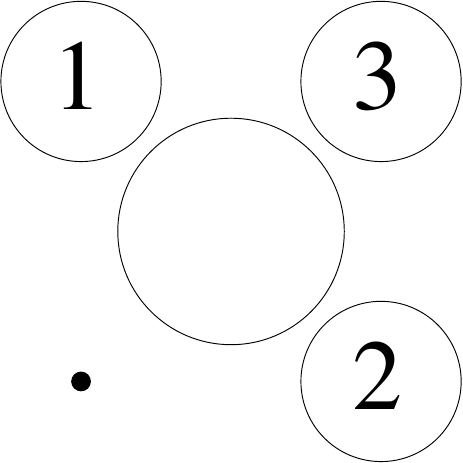}
\\

\vspace{0.5cm}
$\displaystyle
\frac{3+11q+5q^2+5q^3}{96(1+q)(1+q+q^2)}$
\end{minipage}

\vspace{0.4cm}

\begin{minipage}{0.325\linewidth}
\centering
\includegraphics[width=0.6\linewidth]{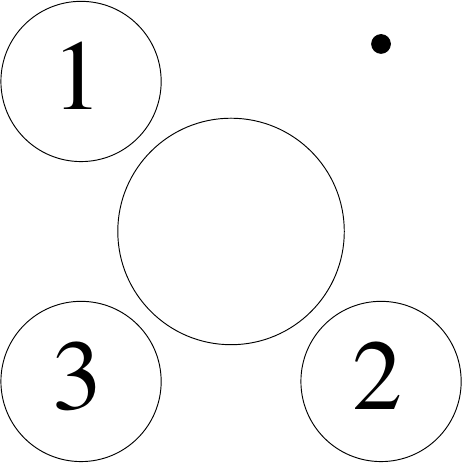}
\\

\vspace{0.5cm}
$\displaystyle
\frac{5+5q+11q^2+3q^3}{96(1+q)(1+q+q^2)}$
\end{minipage}
\begin{minipage}{0.32\linewidth}
\centering
\includegraphics[width=0.6\linewidth]{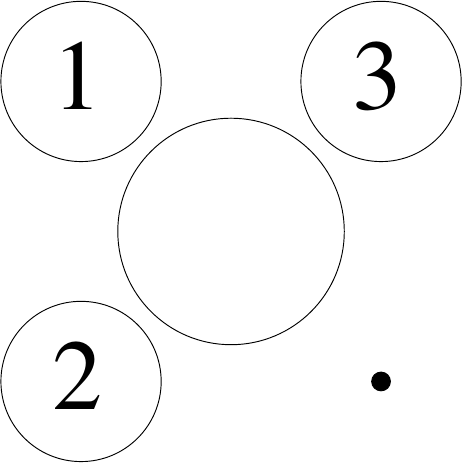}
\\

\vspace{0.5cm}
$\displaystyle
\frac{3+9q+9q^2+3q^3}{96(1+q)(1+q+q^2)}$
\end{minipage}
\begin{minipage}{0.32\linewidth}
\centering
\includegraphics[width=0.6\linewidth]{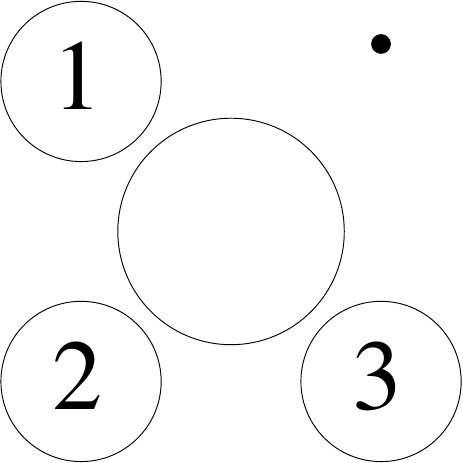}
\\

\vspace{0.5cm}
$\displaystyle
\frac{1+7q+7q^2+9q^3}{96(1+q)(1+q+q^2)}$
\end{minipage}
\caption{
\label{fig:ring4}
Stationary probabilities
for a system with $4$ sites 
(numbered clockwise around the ring). 
The particles are labelled with their type; 
the single hole may equivalently be seen as a particle of 
type $4$. There are $6$ possible configurations,
up to rotation. Note the symmetry between 
$q$ and $1/q$; the probabilities remain unchanged
under replacing $q$ by $1/q$ and reversing the 
order of the particles (since replacing $q$ by $1/q$
is equivalent to exchanging left and right and
multiplying time by a factor $1/q$; the time-change
has no effect on the stationary distribution). 
Note also that all expressions are 
equal when $q=1$, since this gives the symmetric
exclusion process whose equilibrium is uniform
over all configurations. }
\end{figure}

As an example, we can consider the case 
of the $4$-type system on a ring of size $4$.
Corollary \ref{cor:Ldenominators} 
gives a common denominator of 
$96(1+q)^2(1+q+q^2)=96(1+3q+4q^2+3q^3+q^4)$. 
However, in fact we can dispense with one of the factors 
of $1+q$; the stationary probabilities 
as a function of $q$ are shown in Figure \ref{fig:ring4}.

The extra symmetries of the multiline diagram may be easier to explore in the context of the alternative queueing construction mentioned in Section \ref{sec:alternative}, 
where we do not need to treat particles differently
according to whether a service is available to them
at the time they enter the queue. 

\subsection{Clustering}
\label{sec:main-clustering}
Our final result explains the ``clustering" phenomenon
visible in Figure \ref{fig:convoys}, showing 
samples from the stationary distribution 
of the ASEP with $1000$ distinct labels $1,2,\dots, 1000$ on the ring with $1000$ sites, for different values of $q$. 
A striking feature of the configurations is 
the appearance of long strings of nearby particles with similar labels (``convoys"). 
In this section we explain some aspects of that picture,
by considering an appropriate local limit of stationary distributions as $L\to\infty$.
The analysis will rely in an integral way on the queueing construction of the stationary distribution
of the 2-type system on $\ZZ$ which we develop in
Section \ref{sec:queueing} (see Theorem \ref{thm:Nstatline}). 

Consider the ASEP on $\ZZ_L$, with sites now written for convenience as $-\lfloor\frac L2\rfloor, \dots,$ $-1,$ $0,$ $1,$ $\dots, \lceil\frac L2\rceil-1$
in a cyclic way, and with $L$ particles with distinct types $1$, $2,\dots$, $L$.  (We could equally write the largest-numbered type as $\infty$, i.e.\ a hole.)
The process is irreducible, and its stationary distribution can be constructed
via an ``$(L-1)$-line process" involving repeated applications of Algorithm \ref{multialgo}, as described in Section 
\ref{sec:main-algorithms}.

Let $Y^{(L)}_i$ be the type of the particle at site $i$, in a configuration distributed according to this stationary distribution. We will rescale linearly,
so that all labels lie in $[0,1]$, and also pad the configuration with infinite strings of zeros 
on each side, to define a configuration $W^{(L)}=(W^{(L)}_i, i\in\ZZ)$ in $[0,1]^\ZZ$:
\begin{equation}
\label{WLdef}
W^{(L)}:=
\left(
\dots,0,0,0,\tfrac1L Y^{(L)}_{-\lfloor\frac L2\rfloor},\dots, \tfrac1L Y^{(L)}_{-1},
\tfrac1L Y^{(L)}_{0}, \tfrac1L Y^{(L)}_{1}, \dots, \tfrac1L Y^{(L)}_{\lceil\frac L2\rceil-1},0,0,0,
\dots\right),
\end{equation}
or more precisely, for $i\in\ZZ$,
\[
W^{(L)}_i=\begin{cases}
\tfrac{1}{L}Y^{(L)}_i&\text{for }-\lfloor\frac L2\rfloor\leq i\leq \lceil\frac L2\rceil-1,\\
0&\text{otherwise}
\end{cases}.
\]

The following results describe some aspects of the clustering phenomenon observed
in the configurations in Figure \ref{fig:convoys}:
\begin{theorem}\label{thm:Wlimit}
\begin{itemize}
\item[(a)]
As $L\to\infty$, $W^{(L)}=(W^{(L)}_i, i\in\ZZ)$ converges in distribution 
(with respect to the product topology on $[0,1]^\ZZ$) 
to a limit which we denote by $W=(W_i, i\in\ZZ)$. 
The distribution of $W$ is translation-invariant. 
\item[(b)]
$(W_0, W_1)$ has joint distribution on $[0,1]^2$ with density given by
\[
f(x,y)dx\,dy + f^*(x)\ind(x=y)dx
\]
where
\begin{align}
\label{fxy}
f(x,y)=&
\begin{cases}
1&\text{ for }x<y,\\
2(1-q)(x-y)+q&\text{ for }x>y,
\end{cases}
\\
\nonumber
\\
\label{fstarx}
f^*(x)&=(1-q)x(1-x).
\end{align}
$[$By this we mean that if $A$ is a measurable subset of $[0,1]^2$, then
\[
\PP\big((W_0,W_1)\in A, W_0\ne W_1\big)=\int_A f(x,y)dx\,dy,
\]
and if $B$ is any measurable subset of $[0,1]$, then 
\[
\PP(W_0=W_1, W_0\in B)=\int_B f^*(x)dx.
\,\,\,]
\]

It follows that 
\begin{align}
\nonumber
\PP(W_0<W_1)&=\frac12\\
\label{W1W2}
\PP(W_0=W_1)&=\frac{1-q}6\\
\nonumber
\PP(W_0>W_1)&=\frac{2+q}6.
\end{align}
\item[(c)]
With probability 1, there are infinitely many $k$ such that $W_0=W_k$.
\end{itemize}
\end{theorem}

For each $i$, $W_i$ must have 
uniform distribution on $[0,1]$
(since $W_0$ is the limit in distribution of $Y_0^{(L)}/L$,
and $Y_0^{(L)}$ is uniformly distributed on $\{1,2,\dots, L\}$). 
However, it is definitely not the case that 
the $W_i$ are independent; indeed, two neighbouring values
have positive probability to be equal! Note that the density $f$ of two neighbouring values $(W_0, W_1)$ on 
$x>y$ decays to zero as the diagonal is approached if $q=0$, but not if $q>0$. 


Going beyond part (c), the methods of Section 
\ref{sec:convoys} in fact explain how to describe the 
distribution of the ``convoy" $\{k\geq1: W_k=W_0\}$
conditional on $\{W_0=x\}$ for $x\in[0,1]$ somewhat
explicitly in terms of a random walk on $\ZZ$ whose 
transition probabilities depend on $x$.

In the TASEP case ($q=0$), the process $W$ of Theorem 
\ref{thm:Wlimit}
was considered in a related context by Amir, Angel and Valk{\'o}
\cite{AAV}. Consider a TASEP on $\ZZ$ 
started from an initial state in which site $i$ contains
a particle of type $-i$, for each $i\in\ZZ$; the dynamics
of the process are that each pair of neighbouring particles 
sort themselves into increasing order at rate $1$.
Let $X_i(t)$ be the position at time $t$ of the particle
with label $i$. The limit $U_i=\lim_{t\to\infty}\frac{1}{t}X_i(t)$,
i.e.\ the ``asymptotic speed of particle $i$",
exists with probability $1$ for all $i$ (as follows from a result of 
\cite{MouGui}), and has Uniform$[-1,1]$ distribution.
The \textit{TASEP speed process}
$(U_i, i\in\ZZ)$ was introduced and studied by \cite{AAV}. 
Rescaling by $W_i=(1-U_i)/2$, we get
the process $(W_i, i\in\ZZ)$ with the distribution 
described in Theorem \ref{thm:Wlimit}. The properties (\ref{fxy})-(\ref{W1W2}) for the case $q=0$ are given in Theorem 1.7 of \cite{AAV}, and further results concern for example 
the joint distribution 
of more than two entries. 

Still in the case $q=0$, particularly sharp results for 
stationary distributions on finite rings were then given by Ayyer and Linusson in \cite{AyyerLinussonCorrelations}. Their results include closed-form expressions for ``three-point" probabilities of the form $\PP(Y_0^{(L)}=\ell, Y_1^{(L)}=m, Y_2^{(L)}=n)$, as well as more general ``two-point" probabilities of the form $\PP(Y_0^{(L)}=m, Y_i^{(L)}=n)$. (By taking appropriate limits, one can 
regain some of the formulas from \cite{AAV} for the speed process.)
The proofs involve an intricate combinatorial 
analysis of the multi-line queue construction for the TASEP. 
It would of course be interesting to explore whether an analogous
application of the multi-line constructions presented here could lead
to similar results for multi-point probabilities in the case $q>0$.








\subsection{Related work}
\label{sec:related}
Let us mention some further related work, in addition 
to that discussed above. 
The recursive approach to the construction of multi-type 
particle systems has been extended to a variety of 
different particle systems, including discrete-time TASEPs \cite{MarSchdiscrete}, 
inhomogeneous (or ``multi-rate") versions of 
the multi-type TASEP \cite{AritaMallick, AyyerLinusson, 
Cantini-inhom}, and a variety of zero-range processes
\cite{KMO-ZRP1, KMO-ZRP2, KMO-ZRP-inhom},
and also to the description of joint distributions
of Busemann functions for last-passage percolation
\cite{FanSepp}. 

Connections between the multi-type ASEP on the ring and 
families of symmetric polynomials 
such as Schubert polynomials and Macdonald polynomials
have been studied by various authors, for example by Cantini
in \cite{Cantini-inhom} and in recent work by Corteel, Mandelshtam and Williams
in \cite{CorteelMandelshtamWilliams}; the latter paper involves a description of Macdonald polynomials using objects related to the multi-line diagrams
described in this paper. 

Stationary distributions for the multi-type ASEP on a finite interval with open boundary conditions 
have also been widely studied. Not all sets of boundary rates
are expected to lead to exactly solvable models; however,
a variety of classes of integrable boundary conditions have recently been
established
-- see for example work of Crampe, Finn, Ragoucy and Vanicat
\cite{CrampeFinnRagoucyVanicat2016, FinnRagoucyVanicat2018}, 
and also work of Cantini, Garbali, de Gier and Wheeler
\cite{CdGW,
CGdGW}
where further 
connections to families of orthogonal polynomials are made. 

Systems with closed boundaries on finite or half-infinite intervals 
have recently been considered by Belitsky and Sch\"{u}tz \cite{BelitskySchutz}
(see also \cite{arita2012remarks}); 
as well as describing stationary distributions, 
they obtain duality properties and use them to 
study hydrodynamic limits (including the behaviour of shocks).
The stationary distributions of these systems
are closely related to certain stationary distributions for the system on $\ZZ$ 
which (unlike those considered in this paper) are not translation-invariant; 
instead their projections onto one-type distributions are blocking 
measures in the sense of 
\cite{BraLigMou}. The connections between distributions of this
type and the Mallows measure on permutations 
were previously studied by Gnedin and Olshanskii \cite{GneOls1sided, GneOls2sided}, and recent work by Angel, Holroyd, Hutchcroft and Levy \cite{AngelHolroydHutchcroftLevy} describes
a link between such processes and a model of stable matchings.

Integral formulas for multi-type ASEPs on $\ZZ$ 
with finitely many particles are given by Tracy and Widom in
\cite{TraWidmultiASEP}. 

\subsection{Plan of the paper}
\label{sec:plan}
The rest of the paper is organised as follows. 

In Section \ref{sec:matrices} we introduce the recursive
matrix product construction of Prolhac, Evans and Mallick
\cite{PEM}. We give specific realisations of the relevant 
matrices, which are closely related to Markov transition
matrices for the queueing systems that we use to construct
the multi-type ASEP stationary distribution.

These queueing models are introduced in 
Section \ref{sec:queueing}. We also give results concerning stationary distributions of multi-type ASEPs on $\ZZ$, which may also be considered as main results of the paper in their own right. 

The proofs of Theorems \ref{thm:2algo}, \ref{thm:multialgo}
and \ref{thm:denominators} are then developed in 
Section \ref{sec:mainproofs}.
We cover the particular case $N=2$ in some detail,
with the aim of making the argument in the more general case as easily comprehensible as possible. 

The results concerning stationary distributions 
on $\ZZ$ are proved in Section \ref{sec:Z}; these
are deduced from the corresponding results for processes
on the ring using a rather intricate coupling argument,
which may be of independent interest. 

Those results form a central part of the proof of 
Theorem \ref{thm:Wlimit}, given in Section \ref{sec:convoys}. 
The limit process $W$ of Theorem \ref{thm:Wlimit}
can be identified as a stationary distribution 
for an ASEP on $\ZZ$ whose particle-types are continuous,
and distributed uniformly on $[0,1]$. 
This process is a generalisation of the ``TASEP speed process"
studied by Amir, Angel and Valko in \cite{AAV} 
(see also \cite{FerGonMar} for closely related results). 
The process $W$ can be studied via the projection 
of the continuous-type ASEP onto an $N$-type process; 
much useful information can be extracted already from the 
case $N=2$. The hardest part of the proof is the argument
to establish Theorem \ref{sec:convoys}(c), 
where a subtle argument involving 
stochastic domination between random walks is required.

Finally in Section \ref{sec:alternative} we briefly discuss
an alternative construction involving a modified queueing discipline; 
this would result in a more complicated matrix-product
structure, but a rather simpler and more natural 
formulation in terms of multi-line diagrams and their weights. 

\section{Matrix product framework}
\label{sec:matrices}
In this section we describe the matrix product representation for the stationary distribution of the process on the ring
given by Prolhac, Evans and Mallick \cite{PEM},
which is the starting-point of the proof
of the multi-line construction for the ASEP
stationary distribution.
Our notation will be slightly different from in that paper; since 
we consider the ASEP with jumps left rather than right, the matrices which appear in the matrix product solution
are the transposes of those used in \cite{PEM}. 
This change is convenient since 
now time in the associated queueing systems flows from left to right, and the matrices which appear are closely related to Markov transition matrices. 

Suppose the matrices $\mdelta$, $\mepsilon$ and $\malpha$ satisfy the following relations:
\begin{align}\nonumber
\mepsilon \mdelta - q \mdelta\mepsilon &= (1-q) \mI \\
\label{fundamentalrelations}
\malpha\mdelta &= q\mdelta\malpha \\
\nonumber
\mepsilon\malpha &= q\malpha\mepsilon
\end{align}
At (\ref{adedef}) below, we'll give specific examples of appropriate matrices (which will be infinite-dimensional) 
$\malpha$, $\mdelta$, $\mepsilon$. 

Now we define matrices $X_n^{(N)}$ for $n=1,2,\dots,N,\infty$. 
Recursively, the matrices $X_n^{(N)}$ are defined as sums of tensor
products of matrices $X_{j}^{(N-1)}$.
First let $X_1^{(1)}$ and $X_\infty^{(1)}$ be scalars and equal to 1. 

Now for $N\geq 2$, define
\begin{equation}\label{Xdef}
X_j^{(N)} = \sum_{m=1,\dots,N-1,\infty} a_{j,m}^{(N)}\otimes X_{m}^{(N-1)}
\end{equation}
where the matrices $a_{m,j}^{(N)}$ are given by 
\begin{alignat}{2}
\nonumber
a_{\infty,\infty}^{(N)}&=\mI^{\otimes(N-1)}& \\
\nonumber
a_{m, \infty}^{(N)}&=\mI^{\otimes(m-1)}\otimes \mepsilon\otimes\mI^{\otimes(N-m-1)}
&\text{ for }& m\leq N-1\\
\nonumber
a_{\infty, n}^{(N)}&=\malpha^{\otimes(n-1)}\otimes\mdelta\otimes\mI^{\otimes(N-n-1)}
&\text{ for }& n\leq N-1\\
\label{adef}
a_{m,n}^{(N)}&=
\malpha^{\otimes(n-1)}\otimes \mdelta \otimes \mI^{\otimes(m-n-1)}\otimes\mepsilon
\otimes \mI^{\otimes(N-m-1)}
&\text{ for }& n < m \leq N-1\\
\nonumber 
a_{n,n}^{(N)}&=\malpha^{\otimes(n-1)}\mI^{\otimes(N-n)}&\text{ for }& n\leq N-1\\
\nonumber
a_{\infty, N}^{(N)}&=\malpha^{\otimes(N-1)}&\\
\nonumber
a_{m,n}^{(N)}&=0 &\text{ for }& m<n<\infty
\end{alignat}

The dynamics of the ASEP on the ring $\ZZ_L$ preserve the number of particles of each type. 
Consider the process with $k_n$ particles of type $n$, for $n=1,2,\dots, N$,
where $k_n>0$ for all $n$ and also $k_1+k_2+\dots+k_N<L$. 

\begin{theorem}[Prolhac, Evans and Mallick \cite{PEM}]\label{PEMtheorem}
The stationary distribution of the $N$-type TASEP on $\ZZ_L$, 
with $k_n$ particles of type $n$ for $n=1,2,\dots, N$, is given by 
\begin{equation}\label{matrixproductform}
\nu^{(N,L)}_{k_1, \dots, k_N}(\eta_1, \eta_2, \dots, \eta_L)
=
\frac{1}{Z^{(N,L)}_{k_1, \dots, k_N}}
\operatorname{trace}\left(X^{(N)}_{\eta_1} X^{(N)}_{\eta_2} \dots X^{(N)}_{\eta_L}\right),
\end{equation}
where $Z^{(N,L)}_{k_1,\dots, k_N}$ is a normalizing constant
chosen such that the
sum of the right-hand side over all configurations with the particle counts $k_1, \dots, k_N$
is 1. 
\end{theorem}

Matrices satisfying the relations (\ref{fundamentalrelations})
can be realised in many ways. We give a particular version with direct links to 
the queueing interpretations that we will develop:
\begin{gather}
\nonumber
\malpha=
\begin{pmatrix} 
1 & 0 & 0 & 0 & \ldots \\
0 & q & 0 & 0 &\ldots \\
0 & 0 & q^2 & 0 &\ldots \\
0 & 0 & 0 & q^3 & \ldots \\
\vdots & \vdots & \vdots & \vdots & \ddots 
\end{pmatrix}
,
\,\,\, 
\mepsilon=\begin{pmatrix} 
0 & 1 & 0 & 0 & \ldots \\
0 & 0 & 1 & 0 & \ldots \\
0 & 0 & 0 & 1 & \ldots \\
0 & 0 & 0 & 0 & \ldots \\
\vdots & \vdots & \vdots & \vdots & \ddots 
\end{pmatrix},\\ 
\label{adedef}
\mdelta=\begin{pmatrix} 
0 & 0 & 0 & 0 &\ldots \\
1-q & 0 & 0 & 0 &\ldots \\
0 & 1-q^2 & 0 & 0 &\ldots \\
0 & 0 & 1-q^3 & 0 & \ldots \\
\vdots & \vdots & \vdots & \vdots & \ddots 
\end{pmatrix}.
\end{gather}

Note immediately that $\mepsilon$ is a stochastic matrix
(that is, a matrix whose entries are non-negative and whose
row-sums are all equal to $1$), 
and also $\mdelta+\malpha$ is
also a stochastic matrix. Of course, so is 
the identity matrix $\mI$.  
The rows and columns of all these matrices are considered to be indexed by $\ZZ_+=\{0,1,2,\dots\}$. 
The indices can be seen as queue lengths, and the matrices are interpreted as transition matrices in a queueing process evolving over time.

Each matrix $a^{(N)}_{m,j}$ is a tensor product of $N-1$ of the fundamental matrices $\malpha$, $\mdelta$, $\mepsilon$, $\mI$, and can be seen as indexed
by vectors in $\ZZ_+^{N-1}$, representing
queue-lengths of customers of types $1,2,\dots, N-1$
in an $(N-1)$-type queueing process. 
The quantity $a^{(N)}_{m,n}$ will represent 
the weight of a transition associated with the 
arrival of a customer of type $m$ and departure of a customer of type $n$ (with $\infty$ representing no customer). 

Since the tensor products $a^{(N)}_{m,n}$ have order $N-1$, 
it follows from (\ref{Xdef}) that each matrix $X^{(N)}_n$ is 
a sum of tensor products of order $1+2+\dots+(N-1)=N(N-1)/2$.
In fact, the non-zero contributions to $X^{(N)}_n$ 
are all terms of the form
\[
a^{(1)}_{m_1, m_2}\otimes a^{(2)}_{m_2, m_3}\otimes
\dots\otimes a^{(N-1)}_{m_{N-2}, m_{N-1}}
\otimes a^{(N)}_{m_{N-1}, n},
\]
where $m_r\in\{1,\dots, r, \infty\}$ for $r=1,\dots, N-1$.
This quantity represents the weight of a transition
in a system of $N$ queues in series, associated to 
the arrival of a customer of type $m_1$ in the first queue,
a transfer of a customer of type $m_r$ from queue $r-1$
to queue $r$ for $2\leq r\leq N-1$, and a departure
of a customer of type $n$ from queue $N-1$
(and the value $\infty$ indicates the absence of a customer).

\section{Queueing construction}
\label{sec:queueing}
\subsection{The multi-type ASEP on $\ZZ$}
\label{sec:Zdefinition}
The multi-type ASEP on the whole line $\ZZ$ is defined
analogously to the process on the ring $\ZZ_L$. The $N$-type system is a continuous-time Markov process
with state-space $\{1,2,\dots,N,\infty\}^{\ZZ}$. 
For a configuration $\eta=(\eta_i, i\in \ZZ)$, say
that $\eta_i$ is the type of the particle at site $i$.
The dynamics are as follows: if $\eta(i)>\eta(i+1)$, 
then the values $\eta(i)$ and $\eta(i+1)$ are
exchanged at rate $1$. If instead $\eta(i)<\eta(i+1)$, 
then the values $\eta(i)$ and $\eta(i+1)$ are exchanged at rate $q$.

In the case $N=1$, the ergodic translation-invariant stationary distributions
are all Bernoulli product measures, in which each site contains 
a particle (type $1$) with probability $\lambda$, and otherwise a hole (where $\lambda\in[0,1]$). (There are also 
non-translation-invariant stationary distributions, 
the \textit{blocking measures} considered for example by 
\cite{BraLigMou}; these are concentrated on configurations
with only finitely many holes to the left of the origin and only finitely many particles to the right of the origin.)

As shown in \cite{FerrariKipnisSaada}, for given  
$\lambda_1, \dots, \lambda_N$ with $\lambda_1+\dots+\lambda_N<1$,
there is a unique ergodic translation-invariant stationary 
distribution for the $N$-type ASEP on $\Z$ in which the intensity 
of type-$n$ particles is $\lambda_n$ for $1\leq n\leq N$. 
We denote it by $\nu^{(N)}_{\lambda_1, \dots, \lambda_N}$. 

For each $n=1,\dots, N$, we can consider a projection
under which types $r\leq n$ are considered ``particles" and types
$r>n$ are considered ``holes". Under this projection, the $N$-type
ASEP becomes a one-type ASEP, and so in particular the image
of $\nu^{(N)}_{\lambda_1,\dots, \lambda_N}$ is Bernoulli 
product measure with intensity $\lambda_1+\dots+\lambda_r$. 
However, although all these projections are product measures,
$\nu^{(N)}_{\lambda_1,
\dots, \lambda_N}$ is not itself a product measure!
In the case $q=0$, this stationary measure on $\ZZ$ 
was constructed in \cite{FerMarmulti}.

\subsection{Single-type queue}\label{sec:basicqueue}
We now define the basic model of a discrete-time queue 
including rejected services, which will be used to 
describe the stationary distribution of the two-type ASEP on $\ZZ$. 
In later sections we consider systems consisting of several such queues 
in series, with multiple types of customer -- these will be used 
to describe multi-type stationary distributions.  We then
describe analogous systems where the ``time" index is cyclic, 
in order to describe the stationary distributions of systems on the ring $\ZZ_L$. 

The queue is Markovian. Write $A_i=1$ if there is an arrival at time-step $i$,
and $A_i=\infty$ otherwise. Write $S_i=1$ if there is a service available at 
time-step $i$, and $S_i=\infty$ otherwise. 
The processes of arrivals and of services are
independent Bernoulli processes, with rates $\lambda$ and $\mu$ respectively.
That is, at each time-step, an arrival occurs with probability
$\lambda$, and then independently a service is available with probability $\mu$
(with independence between different time-steps). 

Suppose the queue-length at the start of the time-step is $k$. There
are four possibilities:
\begin{itemize}
\item No arrival occurs, no service available, 
with probability $(1-\lambda)(1-\mu)$. The queue-length remains $k$.
\item Arrival occurs and service available, with probability $\lambda\mu$. A departure occurs, and the queue-length remains $k$.
\item Arrival occurs, no service available, 
with probability $\lambda(1-\mu)$. The queue-length increases to $k+1$.
\item No arrival occurs, service is available, with probability $(1-\lambda)\mu$. 
With probability $1-q^k$, a departure occurs, and the queue-length
goes down to $k-1$. With probability $q^k$, an \textit{unused service} occurs, 
and the queue-length remains $k$. 
\end{itemize}

Note that an unused service is allowed to occur only if no arrival has occurred
at that time-slot. We can imagine the service mechanics as follows.
When a service is available, it is offered to each customer in turn.
Each one in turn accepts it with probability $1-q$ and rejects it with probability $q$, 
except that a customer who has just arrived must always accept.  As soon as a customer accepts the service, that customer departs and we stop. 
If all $k$ of the customers reject the service 
(with probability $q^k$ if no arrival has occurred) 
then it remains unused.

The transition matrix 
of the queue-length process is given by
\begin{equation}\label{Pdef}
(1-\lambda)(1-\mu)\mI + \lambda\mu\mI + \lambda(1-\mu)\mepsilon + (1-\lambda)\mu\big(\mdelta+\malpha)
\end{equation}
where $\malpha$, $\mdelta$, $\mepsilon$ are given in (\ref{adedef}). 
The four terms in (\ref{Pdef}) correspond to the four possibilities for
the evolution of a time-step listed above.
The term in $\mdelta$ corresponds to transitions in which a departure but no arrival occurs.
$\malpha$ corresponds to transitions in which no arrival occurs and a service is offered but unused.
$\mepsilon$ corresponds to transitions in which arrival occurs but no service is offered.
Finally, the terms with $\mI$ correspond to transitions where either there is arrival and service,
or there is no arrival and no service is offered. 

For stability we assume $\lambda<\mu$. In that case, 
the queue-length process is positive recurrent, and
there is a unique equilibrium version. 

Note that at each time $i\in\ZZ$, one of three possibilities occurs:
a departure, an unused service, or no available service. 
As at (\ref{2statring}),
we define a departure process $D=(D_i, i\in\ZZ)$ by
\begin{equation}\label{2stat}
D_i =\begin{cases}
1&\text{if a departure occurs at time $i$}\\
2&\text{if an unused service occurs at time $i$}\\
\infty&\text{if no service is available at time $i$}
\end{cases}.
\end{equation} 
(Note that $D_i=\infty$ if and only if $S_i=\infty$.)
Let $Q_i$ be the number of customers in the queue at the beginning of time-step $i$. 
Then we have the following recursive formula:
\begin{equation}\label{qrec}
Q_{i+1}=Q_i + \ind(A_i=1) - \ind(D_i=1).
\end{equation}


\begin{theorem}\label{thm:2statline}
Consider the queue run in equilibrium, with departure process $D$
defined by (\ref{2stat}).
The configuration $D_i, i\in\ZZ$
is distributed according to $\nu^{(2)}_{\lambda, \mu-\lambda}$
(the unique ergodic stationary distribution of the two-type ASEP
on $\ZZ$
with parameter $q$ and with density $\lambda$ of first-class
particles and $\mu-\lambda$ of second-class particles).
\end{theorem}
Our proof of this result will be based on the 
matrix product representation in Theorem 
\ref{PEMtheorem}.

\subsection{Multi-type queues}
Next, we consider the extension of the queue described above 
to a multi-class queue with priorities.
The queue will contain $N-1$ classes (or types) of customer, labelled $1,2,\dots, N-1$.
The lower the number of the class, the higher the priority.

The state of the queue at a given time $i$ is now a 
vector $(Q^{(n)}_i, 1\leq n\leq N-1)$ with 
$N-1$ entries; the $n$th entry $Q^{(n)}_i$ denotes the number of customers of 
class $n$ present in the queue at the beginning of time-step $i$.

At each time-step $i$, at most one customer arrives (with some given class).
Write $A_i=n$ if a customer of type $n$ arrives, and $A_i=\infty$ if there is no 
arrival. Initially we don't specify the arrival process, but concentrate on
describing the action of the queueing server.

As before, the process of available services is a Bernoulli process of rate $\mu$;
write $S_i=1$ if a service occurs (which happens with probability $\mu$)
and $S_i=\infty$ otherwise. As above, the service will be offered in turn to each 
customer in the queue. This is now done in order of priority; 
the service is offered to each of the first-class customers, then 
to each of the second-class customers, and so on until some accepts it 
or all customers have rejected it. Each customer accepts the service 
with probability $1-q$ and rejects it with probability $q$, with the exception
that if a customer has just arrived at the queue, then that customer will always accept. 

We give two brief examples. Suppose that at the beginning of a time-slot, the 
queue contains $3$ first-class customers,
$1$ second-class customer and $4$ third-class customers. Suppose that 
there is no arrival, and a service is available. Then a departure of a first-class customer
occurs with probability $1-q^3$, a second-class departure occurs with probability
$q^3-q^4$, a third-class departure occurs with probability $q^4-q^8$ and 
an unused service occurs with probability $q^8$. 

Suppose instead that an arrival of a second-class customer occurs, increasing the number
of second-class customers to 2. Now a first-class customer departs with probability 
$1-q^3$, and a second-class customer departs with probability $q^3$. It is impossible
for a third-class departure or an unused service to occur. 

Note the particular case $q=0$. Here it is always a customer of the highest
priority-type present who departs, and a service is unused if and only if the 
queue is empty. 

As at (\ref{multistatring}), generalising (\ref{2stat}), 
define the departure process $D=(D_i, i\in\ZZ)$ by
\begin{equation}\label{multistat}
D_i =\begin{cases}
n&\text{if a departure of type $n$ occurs at time $i$, for $n\in\{1,2,\dots,N-1\}$}\\
N&\text{if an unused service occurs at time $i$}\\
\infty&\text{if no service is available at time $i$}
\end{cases}.
\end{equation}
As at (\ref{qrec}), we have a recursion for the queue-length process. For each $i$ and 
for each $r=1,2,\dots, N-1$,
\[
Q^{(n)}_{i+1}=Q_i^{(n)}+\ind(A_i=n)-\ind(D_i=n).
\]

\subsection{Multi-type stationary distribution}
We now explain how to construct stationary distributions for the multitype ASEP
on $\ZZ$ recursively. 
The departure process of a queue with $N-1$ types, whose arrival process corresponds
to the stationary distribution for the $(N-1)$-type ASEP, is used to give a stationary 
distribution for the $N$-type ASEP.

Fix $\lambda_1,\lambda_2,\dots,\lambda_N$ with $\lambda_1+\lambda_2+\dots+\lambda_N<1$.

Consider a priority queue as above with $N-1$ types, whose arrival process 
$A_i, i\in\ZZ$ is
distributed according to $\nu_{\lambda_1, \dots, \lambda_{N-1}}^{(N-1)}$, 
the ergodic stationary distribution of the $(N-1)$-type ASEP with 
density $\lambda_r$ of customers of type $r$, for $r=1,2,\dots,N-1$, and whose service
process is a Bernoulli process with rate $\lambda_1+\lambda_2+\dots+\lambda_n$
(independent of the arrival process).

\begin{theorem}\label{thm:Nstatline}
Let $D$ be the departure process of the priority queue
with $N-1$ classes,
defined at (\ref{multistat}).
The $N$-type configuration $D_i, i\in\ZZ$ 
is distributed according to $\nu_{\lambda_1, \dots, \lambda_{N}}^{(N)}$,
the unique ergodic stationary distribution of the $N$-type ASEP on $\ZZ$
with parameter $q$ and with density $\lambda_r$ of particles of class $r$,
for $r=1,2,\dots,N$.
\end{theorem}

Hence the stationary distribution with $N$ types can be seen as the output of a series
of $N-1$ queues in tandem. The $r$th of the $N-1$ queues contains $r$ types
of customer in its arrival process. Its departure process also contains $r$ types, 
to which we add an $(r+1)$st type corresponding to unused services. 

\subsection{``Queues'' on the ring: two-type}
To construct the stationary distribution of the 
ASEP on $\ZZ_L=\{0,1,\dots,L-1\}$ (the ring with $L$ sites) we consider ``queues''
in which the time is cyclic; i.e.\ we replace the time index $\ZZ$ by $\ZZ_L$. All addition and 
subtraction is to be understood modulo $L$ 

In this section we cover the two-type ASEP.  The numbers of first-class and
of second-class particles are conserved by the dynamics.  For each
$k_1$ and $k_2$ with $k_1+k_2<L$, the set of configurations with
$k_1$ first-class and $k_2$ second-class particles (and hence
$L-k_1-k_2$ holes) forms a single communicating class; there is a
unique stationary distribution concentrated on such a set of configurations.

The construction of the two-type stationary distribution uses a
queue with one type of customer. We will have $k_1$ arrivals and $k_1+k_2$ services. 

The set of times at which arrivals occur is chosen uniformly from all subsets
of $\ZZ_L$ of size $k_1$, and the set of times at which services are available is
chosen uniformly from all subsets of $\ZZ_L$ of size $k_1+k_2$,
independently of arrivals. Write $A_i=1$ if there is an arrival at
time $i$ and $A_i=\infty$ if not. Write $S_i=1$ if there is a service available
at time $i$ and $S_i=\infty$ if not.

Analogously to the queue described in Section \ref{sec:basicqueue}, 
we want the following rules. If no service is available, then no 
departure occurs. If an arrival and an available service both occur, 
then a departure must occur. Finally,
if a service is available but an arrival does not occur, 
there may be either a departure or an unused service. 

Formally, suppose we are given arrival and service processes
$A$ and $S$ in $\{1,\infty\}^{\ZZ_L}$, 
Then we say that a 
queue-length process $Q=(Q_i, i\in\ZZ_L)\in\ZZ_+^{\ZZ_L}$
is \textit{valid}, 
and write $(A,S,Q)\in\cR^{(2)}$, 
if there exists a departure process $D=(D_i, i\in\ZZ_L)
\in\{1,2,\infty\}^{\ZZ_L}$ satisfying the following properties:
\begin{align}
\nonumber
Q_{i+1}-Q_i &\,\,=\,\, \ind(A_i=1)-\ind(D_i=1)\\
\label{R2def}
S_i=\infty &\Longrightarrow D_i=\infty\\
\nonumber
S_i=1 &\Longrightarrow D_i<\infty \text{ and } D_i\leq A_i.
\end{align}
If such a $D$ exists it is unique.
(The first line determines when $D_i=1$ and the next two 
lines together determine when $D_i=\infty$.) 
For $(A,S,Q)\in\cR^{(2)}$, we write $D(A,S,Q)$ for the unique 
$D$ satsifying the properties in (\ref{R2def}).

Now, given $A$ and $S$, define a weight function
on valid queue-length processes as follows.
For valid $Q$,
define the 
weight $w_i(Q|A,S)$ associated to site $i\in\ZZ_L$ by
\begin{equation}\label{weightdef}
w_i(Q|A,S)=\begin{cases} 
1&\text{if } S_i =\infty, D_i=\infty, Q_{i+1}=Q_i +\ind(A_i=1) \\
1&\text{if } S_i =1, D_i=1, A_i =1, Q_{i+1}=Q_i\\
1-q^{Q_i }&\text{if } S_i =1, D_i=1, A_i =\infty, Q_{i+1}=Q_i -1\\
q^{Q_i }&\text{if } S_i =1, D_i=2, A_i =\infty, Q_{i+1}=Q_i
\end{cases},
\end{equation}
where $D=D(A,S,Q)$.
(It's straightforward to check that if the properties
in (\ref{R2def}) are satisfied, then exactly one of the 
cases in (\ref{weightdef}) occurs.)
Now define the weight of the whole process $Q$ by 
$
w(Q|A,S)=\prod_{i\in\ZZ_L} w_i(Q|A,S).
$
Given $A$ and $S$, we now take 
the probability of the queue-length process $Q$ to be proportional to $w(Q|A,S)$,
by
\begin{equation}\label{PQASdef}
P(Q|A,S)=\frac{w(Q|A,S)}{\sum_{Q'} w(Q'|A,S)}.
\end{equation}
It's easily seen that the denominator is finite as long as $q<1$. 
In fact, we will show later (Lemma \ref{lem:denominator2})
that this normalizing constant depends only on $k_1$ and $k_2$, and not on 
particular $A$ and $S$. 

The weight has the following interpretation. 
At each time slot where a service occurs,
this service is offered to the customers in turn.
Each customer accepts an offer with probability $1-q$ 
and rejects it with probability $q$, with the exception 
that any customer who has just arrived at the queue must 
accept it. The values of $w_i$ in the second, third, and fourth lines
of (\ref{weightdef}) above are the corresponding
probabilities for observing a departure or an unused service, 
taking into account the current composition of the queue and the 
information about whether an arrival has occurred.

Now we use the weight $w$ to give a distribution on departure configurations $D$; namely, choose $A$ and $S$ uniformly, as described above; choose $Q$ in proportion to $P(Q|A,S)$
(equivalently, in proportion to $w(Q|A,S)$); finally, take 
$D=D(A,S,Q)$. 

\begin{theorem}\label{thm:2statring}
The distribution of the configuration $D=D_i, i\in\ZZ_L$ 
induced by the weight (\ref{weightdef}) on queue-length processes
is the stationary distribution of the two-type ASEP
on $\ZZ_L$ with $k_1$ first-class and $k_2$ second-class 
particles.
\end{theorem}

\subsection{``Queues'' on the ring: multi-type}
In this section we construct the stationary distribution for 
the ASEP on $\ZZ_L$ with several types of particle,
generalising the result of the previous section for 2-type systems.

The construction is done recursively.
To describe the $N$-type equilibrium, we use the $(N-1)$-type 
equilibrium and a priority ``queue'' (with cyclic time as in the previous
section). 

The arrival process contains $N-1$ types of particle, and holes. 
Let $k_n$ be the number of particles of type $n$. 
Write $A_i=n$ if there is an arrival of type $n$ at time $i$,
and $A_i=\infty$ if there is no arrival. Write $S_i =1$ if there is a service
available at time $i$ and otherwise $S_i =\infty$. 

We choose the process $A$ according to the stationary distribution of an $(N-1)$-type
system with $k_n$ particles of type $n$ for $1\leq n\leq N-1$. 
Independently of the arrivals, the times of potential services
are chosen uniformly from all subsets of $\ZZ_L$ of size $k_1+k_2+\dots+k_N$. 

We now consider queue-length processes $Q^{(n)}_i$, $1\leq n\leq N-1$, $i\in\ZZ_L$
which are consistent with a given configuration of arrivals and services. 
The value $Q^{(n)}_i$ represents the number of customers of type $n$
in the queue at the beginning of time-slot $i$. 

We want the queue to obey the following rules. If no service occurs, 
then no departure can occur. If an arrival and a service both 
occur, then a departure must also occur, and the 
departing customer must have type no larger than that of the arrival.
If a service occurs but an arrival does not occur, there may be either
a departure or an unused service.

Formally, suppose we are given arrival and service processes
$A\in\{1,2,\dots,N-1,\infty\}^{\ZZ_L}$ and $S$ in $\{1,\infty\}^{\ZZ_L}$, 
Then we say that a 
queue-length process $Q=(Q_i^{(n)}, i\in\ZZ_L, 1\leq n\leq N-1)
$
is \textit{valid}, 
and write $(A,S,Q)\in\cR^{(N)}$, 
if there exists a departure process $D=(D_i, i\in\ZZ_L)
\in\{1,2,\dots,N,\infty\}^{\ZZ_L}$ satisfying the following properties:
\begin{align}
\nonumber
Q^{(n)}_{i+1}-Q^{(n)}_i &\,\,=\,\, \ind(A_i=n)-\ind(D_i=n)
\text{ for all } n\in\{1,2,\dots,N-1\}
\\
\label{RNdef}
S_i=\infty &\Longrightarrow D_i=\infty\\
\nonumber
S_i=1 &\Longrightarrow D_i<\infty \text{ and } D_i\leq A_i.
\end{align}
If such a $D$ exists it is unique.
(The first line determines when $D_i=n$ for each $n=1,2,\dots,N-1$ 
and the next two 
lines together determine when $D_i=\infty$.) 
For $(A,S,Q)\in\cR^{(N)}$, we write $D(A,S,Q)$ for the unique 
$D$ satsifying the properties in (\ref{RNdef}).

Now, given $A$ and $S$, define a weight function
on valid queue-length processes as follows.

For $0\leq n\leq N-1$, we write 
$Q^{(\leq n)}(i)=\sum_{r=1}^n Q^{(r)}(i)$, for the number of customers of type $n$ or below in the queue at the beginning 
of time-step $i$. (Vacuously $Q^{(\leq 0)}(i)=0$ for all $i$). 
Write also $e_r=(e_r^{(1)}, \dots, e_r^{(N-1)})$
for the $r$th basis vector, with $e_r^{(n)}=\ind(n=r)$. 

Then for valid $Q$, define the 
weight $w_i(Q|A,S)$ associated to site $i\in\ZZ_L$ by
\begin{equation}\label{multiweightdef}
w_i(Q|A, S)=\begin{cases} 
1&\text{if } S_i =\infty, D_i=\infty, 
A_i=\infty, Q_{i+1}=Q_i\\
1&\text{if } S_i =\infty, D_i=\infty, 
\\
&\phantom{\text{if}  S_i}
A_i<\infty, Q_{i+1}=Q_i+e_{A_i}\\
q^{Q^{(\leq n-1)}(i)}(1-q^{Q^{(n)}(i)})
&\text{if } S_i=1, D_i=n, A_i=\infty, Q_{i+1}=Q_i-e_n \\
q^{Q^{(\leq n-1)}(i)}(1-q^{Q^{(n)}(i)})
&\text{if } S_i=1, D_i=n, 
\\
&\phantom{\text{if}  S_i}
n<A_i<\infty, Q_{i+1}=Q_i-e_n+e_{A_i} \\
q^{Q^{(\leq n-1)}(i)}
&\text{if } S_i=1, D_i=n, A_i=n, Q_{i+1}=Q_i\\
q^{Q^{(\leq N-1)}(i)}
&\text{if } S_i=1, D_i=N, A_i=\infty, Q_{i+1}=Q_i
\end{cases},
\end{equation}
where $D=D(A,S,Q)$.
(Again, one can verify that if the properties
in (\ref{RNdef}) are satisfied, then exactly one of the 
cases in (\ref{multiweightdef}) occurs.)
Now define the weight of the whole process $Q$ by 
$
w(Q|A,S)=\prod_{i\in\ZZ_L} w_i(Q|A,S).
$

Given $A$ and $S$, the probability 
$P(Q|A,S)$ of the queue-length process $Q$ is now taken to be
proportional to $w(Q|A,S)$, just as at (\ref{PQASdef}). 
(We will show in Lemma \ref{lem:denominatorN}
that the denominator $\sum w(Q|A,S)$ is again finite, and indeed depends on $A$ and $S$ only through  
the particle counts $k_1, \dots, k_N$.)

The weight has the following interpretation. 
At each time slot where a service occurs,
this service is offered to each of the customers, in order of priority
(starting with those of type 1, then those of type 2, and so on). 
Each customer accepts an offer with probability $1-q$ 
and rejects it with probability $q$, with the exception 
that any customer who has just arrived at the queue must 
accept it. The values of $w_i$ above are the corresponding
probabilities for observing a particular type of departure or unused service, taking into account the current composition of the queue and the type of arrival (if any).

Then the weight $w$ yields a distribution on departure configurations
$D$; namely, choose $A$ from the $(N-1)$-type stationary distribution,
and independently choose $S$ uniformly, as described above; 
choose $Q$ in proportion to $P(Q|A,S)$
(equivalently, in proportion to $w(Q|A,S)$; finally, take 
$D=D(A,S,Q)$.

\begin{theorem}\label{thm:Nstatring}
The distribution of the configuration $D=D_i, i\in\ZZ_L$ 
induced by the weight (\ref{multiweightdef}) on queue-length processes
is the stationary distribution of the $N$-type ASEP
on $\ZZ_L$ with particle counts $k_1,\dots, k_N$. 
\end{theorem}

\section{Proofs of main results for systems on $\ZZ_L$}
\label{sec:mainproofs}
In this section, we prove Theorem \ref{thm:2statring}
and Theorem \ref{thm:Nstatring} concerning the construction
of the ASEP on $\ZZ_L$ (in the $2$-type and general $N$-type cases 
respectively), and then Theorem \ref{thm:2algo} and
Theorem \ref{thm:multialgo} justifying Algorithms
\ref{2algo} and \ref{multialgo}.

\subsection{Proofs of Theorem \ref{thm:2statring} and 
Theorem \ref{thm:Nstatring}}
\label{sec:ringproofs}
We start with the 2-type case of Theorem \ref{thm:2statring}.
We explain this case somewhat thoroughly, 
so as to indicate as clearly as possible the extension of 
the argument to $N$ types, where the notation is more complicated
and fewer details will be included. 
%
The form of the matrix product solution in the case $N=2$ is 
particularly simple. From (\ref{Xdef}) and (\ref{adef}), we obtain
\begin{align}
\nonumber
X_1^{(2)}&=I+\mdelta,\\
\label{twotypematrices}
X_2^{(2)}&=\malpha,\\
\nonumber
X_\infty^{(2)}&=I+\mepsilon.
\end{align}
We work specifically with the forms of $\malpha$, $\mepsilon$ and $\mdelta$ 
given at (\ref{adedef}). 

\begin{lemma}\label{lemma:matrixweights}
$ $
\begin{itemize}
\item[(i)]
Suppose $(A,S,Q)\in\cR^{(2)}$, 
and let $D=D(A,S,Q)$.
Then for all $i$,
\[
w_i(Q|A,S)=\left(X_{D_i}^{(2)}\right)_{Q_i, Q_{i+1}}.
\]
\item[(ii)]
Given $Q$ and $D$, if
$\prod_i \left(X_{D_i}^{(2)}\right)_{Q_i, Q_{i+1}}>0$,
then there exists a unique pair $A$, $S$ such 
that $(A,S,Q)\in\cR^{(2)}$ and $D=D(A,S,Q)$. 
\end{itemize}
\end{lemma}
\begin{proof}
Note that the non-zero entries of the matrices 
$X_r^{(2)}$, $r=1,2,\infty$, are precisely the 
diagonal and super-diagonal entries of $X_\infty^{(2)}=I+\epsilon$,
the diagonal and sub-diagonal entries of $X_1^{(2)}=I+\delta$, 
and the diagonal entries of $X_2^{(2)}=\alpha$. Then 
there is an exact correspondence between the four cases
\begin{align}
\nonumber
&D_i=\infty, Q_{i+1}-Q_i\in\{0,1\}\\
\label{DQlist}
&D_i=1, Q_{i+1}-Q_i=0\\
\nonumber
&D_i=1, Q_{i+1}-Q_i=-1\\
\nonumber
&D_i=2, Q_{i+1}-Q_i=0
\end{align}
and the four lines of (\ref{weightdef}).
For part (i), we can verify that 
when any one of these four cases hold, 
the relevant matrix entry, namely the $(Q_i, Q_{i+1})$ entry
in $X^{(2)}_{D_i}$, equals the weight defined in (\ref{weightdef}).

For part (ii), if we are given $D$, $Q$ such that for each $i$
one of the lines in (\ref{DQlist}) is satisfied, then defining $A_i$ and $S_i$ for each $i$ according to the corresponding line of (\ref{weightdef}) satisfies (\ref{R2def}), so that 
$(A,S,Q)\in\cR^{(2)}$ and $D=D(A,S,Q)$ as required.

Finally, from (\ref{R2def}) it's also easy to see that for a given $D,Q$ there could not be more than one pair $A,S$ such that $(A,S,D,Q)\in\cR^{(2)}$ (since the first line of (\ref{R2def}) 
determines $A$, and the second and third lines determine $S$). 
\end{proof}

We will need one further property,
that the denominator in (\ref{PQASdef}) does not depend on $A$ and $S$.
\begin{lemma}\label{lem:denominator2}
Take any $A$, $S$ with $\sum_i \ind(S_i=1)-\sum_i \ind(A_i=1)=k_2$. Then
\[
\sum_Q w(Q|A,S)=\left(1-q^{k_2}\right)^{-1},
\]
so for any $Q$, $P(Q|A, S)=(1-q^{k_2})w(Q|A,S)$.
\end{lemma}
This property is a central (and perhaps non-obvious) part of
the argument. Its proof will do much of the work needed for the justification of Algorithm \ref{2algo} in Theorem 
\ref{thm:2algo}, and is given below. 
Meanwhile we complete the proof of Theorem \ref{thm:2statring}:

\begin{proof}[Proof of Theorem \ref{thm:2statring}]
Write $P_{k_1,k_2}^{(2,L)}$ for the distribution described
by (\ref{weightdef}) and (\ref{PQASdef}). 

Take any configuration $D$ with $k_1$ first-class and $k_2$ second-class particles. We restrict to $A$ and $S$ with 
$\sum_i \ind(A_i=1)=k_1$ and $\sum_i \ind(S_i=1)=k_1+k_2$. Then
\begin{align*}
P_{k_1,k_2}^{(2,L)}(D)
&=\sum_{\substack{A,S,Q\in\cR^{(2)}:\\D=D(A,S,Q)}} 
P(Q|A, S)\ch{L}{k_1}^{-1} \ch{L}{k_1+k_2}^{-1}
\\
\intertext{(since $(A,S)$ is chosen uniformly from the $\ch{L}{k_1}\ch{L}{k_1+k_2}$ possibilities)}
&\propto \sum_{\substack{A,S,Q\in\cR^{(2)}:\\D=D(A,S,Q)}} w(Q|A, S)
\text{ \,\,\,(by Lemma \ref{lem:denominator2})}\\
&=\sum_{\substack{A,S,Q\in\cR^{(2)}:\\D=D(A,S,Q)}}
\prod_i w_i(Q|A, S)\\
&=\sum_{Q} \prod_i \left(X_{D_i}^{(2)}\right)_{Q_i , Q_{i+1}}
\text{\,\,\,(using Lemma 
\ref{lemma:matrixweights})}\\
&=\operatorname{trace}\prod_i X_{D_i}^{(2)}\\
&\propto \nu_{k_1, k_2}^{(2,L)}(D),
\end{align*}
by Theorem \ref{PEMtheorem}. So
the two distributions $P_{k_1,k_2}^{(2,L)}$ and $\nu_{k_1,k_2}^{(2,L)}$ are the same as required.
\end{proof}

Now we describe how to generalise the argument to the case $N>2$,
to prove Theorem \ref{thm:Nstatring}.

A queue-length process now has $N-1$ entries,
so that $Q_i =(Q^{(1)}(i),\dots, Q^{(N-1)}(i))$,
where $Q^{(n)}(i)$ gives the number of type-$n$
customers in the queue. The arrival and departure 
proceses are elements of $\{1,2,\dots, N-1, \infty\}^{\ZZ_L}$
and of $\{1,2,\dots, N,\infty\}^{\ZZ_L}$ respectively. 

\begin{lemma}\label{lemma:multimatrixweights}
$ $
\begin{itemize}
\item[(i)] 
Suppose $(A,S,Q)\in\cR^{(N)}$,
and let $D=D(A,S,Q)$.
Then for all $i$,
\[
\left(a^{(N)}_{A_i, D_i}\right)_{Q_i, Q_{i+1}}=w_i(Q|A,S).
\]
\item[(ii)]
Given $A$, $D$, $Q$, suppose 
\[
\prod_i\left(a^{(N)}_{A_i, D_i}\right)_{Q_i, Q_{i+1}}>0.
\]
Then there exists a unique $S$ such that
$(A,S,Q)\in\cR^{(N)}$ and $D=D(A,S,Q)$. 
\end{itemize}
\end{lemma}
\begin{proof}
Similarly to the proof of Lemma \ref{lemma:matrixweights},
we need to show that the lines of 
(\ref{multiweightdef})
correspond to the cases where the entries in the tensor 
products defined in (\ref{adef}) are non-zero.
This correspondence will be line-by-line between 
the first six lines of (\ref{adef}) and the six lines of
(\ref{multiweightdef}). We will not go through every case, but will give an example. 
 
First consider the four matrices $I$, $\alpha$, $\delta$ and $\epsilon$ involved in the tensor products in (\ref{adef}), using as before the explicit forms at (\ref{adedef}). The non-negative
entries of $I$ and $\alpha$ are precisely the diagonal entries; 
those of $\delta$ are the entries on the subdiagonal, and
those of $\epsilon$ are those on the superdiagonal. 

Then consider for example the fourth line of (\ref{adef}), 
which is the case of $a_{m,n}^{(N)}$ for $n<m\leq N-1$. 
Inspecting the tensor product, we see that the subdiagonal 
matrix $\delta$ appears in the $n$th component, and the superdiagonal
matrix $\epsilon$ appears in the $m$th component; all the other matrices involved are diagonal. So the only non-zero terms 
$\left(a_{m,n}^{(N)}\right)_{Q_i, Q_{i+1}}$ are 
those where $Q_{i+1}=Q_i-e_n+e_m$. This corresponds to the
fourth line of (\ref{multiweightdef}). We then need to check that the 
weight defined in the fourth line of (\ref{multiweightdef}) 
is the same as the entry in the tensor product. 
The components $r=1,2,\dots,n-1$ of the tensor
product contribute values $\alpha_{Q_i^{(r)},Q_i^{(r)}}=q^{Q_i^{(r)}}$, 
giving $q^{Q_i^{(\leq n-1)}}$ between them. The $n$th component
contributes the value $\delta_{Q_i^{(n)}, Q_i^{(n)}-1}
=1-q^{(Q_i^{(n)}}$. The remaining components all contribute 
the value $1$ from the relevant entries in the matrix $\epsilon$ 
or $I$. Multiplying together, we indeed obtain the weight
defined in the fourth line of (\ref{multiweightdef}) as required.

In a similar way each of the first six lines of (\ref{adef})
accords with the corresponding line of (\ref{multiweightdef}).

So for part (i), we have that if (\ref{RNdef}) is satisfied,
then one of the lines of (\ref{multiweightdef}) holds, 
and then the weight defined by that line is the same as the tensor entry in the corresponding line of 
(\ref{adef}).

For part (ii), suppose we have $A$, $D$, $Q$ such that 
the tensor entry $\left(a^{(N)}_{A_i, D_i}\right)_{Q_i, Q_{i+1}}$
is positive for each $i$. That is, for each $i$, one of the 
lines of $(\ref{adef})$ is positive; then one can verify 
as above that if we take $S_i=1$ if $D_i<\infty$ and $S_i=\infty$ if $D_i=\infty$, the corresponding line of $(\ref{multiweightdef})$ 
also holds, and that this is the only such choice of $S_i$. 
This choice of $S$ then satisfies (\ref{RNdef}) 
(and is the only such choice). 
\end{proof}


The corresponding result to Lemma \ref{lem:denominator2}
is:
\begin{lemma}\label{lem:denominatorN}
Fix any $A$, $S$ with $|\cA_n|=k_n$ for $1\leq n\leq N-1$ 
and $|\cS|=k_1+k_2+\dots+k_N$. Then
\[
\sum_Q w(Q|A,S)=
\left(1-q^{k_2}\right)^{-1}
\dots
\left(1-q^{k_N}\right)^{-1},
\]
so for any $Q$, $P(Q|A, S)=
(1-q^{k_2})\dots(1-q^{k_N})w(Q|A,S)$.
\end{lemma}
Again this result is a central part of the subsequent argument,
and we prove it below. 

\begin{proof}[Proof of Theorem \ref{thm:Nstatring}]
Write $P_{k_1,\dots,k_N}^{N,L}$ for the distribution resulting from 
the weight defined at \ref{multiweightdef}. 
Let $D$ be any configuration with particle counts $k_1,\dots,k_N$.
Note that if $D=D(A,S,Q)$ then $A$ has particle counts
$k_1,\dots, k_{N-1}$ and $\sum_{i} \ind(S_i=1)=k_1+\dots+k_N$. 
We have
\begingroup\allowdisplaybreaks
\begin{align*}
P_{k_1,\dots, k_N}^{(N,L)}(D)
&=
\sum_{\substack{(A,S,Q)\in\cR^{(N)}:\\ D=D(A,S,Q)}}
P(Q|A, S)\nu_{k_1, \dots, k_{N-1}}^{(N-1,L)}(A)\Ch{L}{k_1+\dots+k_N}^{-1}
\\
&\propto
\sum_{\substack{(A,S,Q)\in\cR^{(N)}:\\ D=D(A,S,Q)}}
w(Q|A, S)\nu_{k_1, \dots, k_{N-1}}^{(N-1,L)}(A)
\text{\,\,\, (by Lemma \ref{lem:denominatorN})}
\\
&=
\sum_{\substack{(A,S,Q)\in\cR^{(N)}:\\ D=D(A,S,Q)}}
\prod_i 
\left(a^{(N)}_{A_i , D_i}\right)_{Q_i , Q_{i+1}}
\nu_{k_1, \dots, k_{N-1}}^{(N-1)}(A)
\text{\,\,\, (by Lemma \ref{lemma:multimatrixweights})}
\\
&=
\sum_{A,Q} 
\prod_i 
\left(a^{(N)}_{A_i , D_i}\right)_{Q_i , Q_{i+1}}
\nu_{k_1, \dots, k_{N-1}}^{(N-1)}(A)
\text{\,\,\, (by Lemma \ref{lemma:multimatrixweights} again)}
\\
&=
\sum_A \nu_{k_1, \dots, k_{N-1}}^{(N-1)}(A)
\sum_Q
\prod_i 
\left(a^{(N)}_{A_i , D_i}\right)_{Q_i , Q_{i+1}}
\\
&=\sum_A 
\trace\left(
X_{A_1}^{(N-1)}\dots X_{A_L}^{(N-1)}
\right)
\trace\left(a^{(N)}_{A_0, D_0}\dots a^{(N)}_{A_{L-1}D_{L-1}}\right)
\\
&=
\trace
\left(
\left(\sum_{A_0}a^{(N)}_{A_0D_0}\otimes X^{(N-1)}_{A_0}\right)
\dots
\left(\sum_{A_{L-1}}a^{(N)}_{A_{L-1}D_{L-1}}\otimes X^{(N-1)}_{A_{L-1}}\right)
\right)
\\
&=
\trace\left(X_{D_0}^{(N)}\dots X_{D_{L-1}}^{(N-1)}\right)
\\
&\propto \nu_{k_1,\dots, k_N}^{(N,L)}(D),
\end{align*}%
\endgroup
so that the distributions $P_{k_1,\dots, k_N}^{(N,L)}$
and $\nu_{k_1,\dots, k_N}^{(N,L)}$ are the same as required.
\end{proof}

\subsection{Algorithmic results: Proofs of 
Theorem \ref{thm:2algo} and Theorem \ref{thm:multialgo}}

In this section we complete the 
proofs of Lemmas \ref{lem:denominator2} and \ref{lem:denominatorN}
left over from the last section (hence completing 
the proofs of Theorems \ref{thm:2statring} and \ref{thm:Nstatring}
from that section), and proceed to justify the results
of Theorems \ref{thm:2algo} and \ref{thm:multialgo} that
Algorithms \ref{2algo} and \ref{multialgo} produce
samples from the multi-type stationary distributions.

We start with the $2$-type case. 
Throughout, we fix some arrival process $A$ with $k_1$ arrivals, 
and some service process $S$ with $k_1+k_2$ potential
services. Let $\cS$ be the set of times where service is offered; that is, $\cS=\{i\in\ZZ_L:S_i=1\}$.

We first define a system of 
random marks attached to the service times, 
which will be used to determine which services are used
and which remain unused. Let $B_i , i\in\cS$ be 
i.i.d.\ geometric with parameter $q$; that is,
$B_i$ takes value $b_i\in\{0,1,2,\dots\}$ with probability $q^{b_i}(1-q)$. 
We write $\PP_B$ for the joint law of the $B_i $, 
so for a given vector $(b_{i})_{i\in\cS}$, we have 
$\PP_B(B_i =b_{i} \text{ for all } i)=
\prod_{i\in\cS}q^{b_{i}}(1-q)=(1-q)^{|\cS|}q^{\sum b_{i}}$.

Recall that under our queueing model, if there is no arrival at $i$ and 
the current queue-length $Q_i $ is $k$, the service at $i$ is unused with probability
$q^k$ and used with probability $1-q^k$. We will arrange
that the service at $i$ is used if $B_i <Q_i $, and unused if $B_i \geq Q_i $.
Hence $B_i $ may be interpreted as the number of times that the service at 
$i$ is refused by a customer.
Suppose $Q$ is a valid queue-length process (i.e. $(A,S,Q)\in\cR^{(2)}$).
We will say that $Q$ is compatible with a mark $b_{i}$ at $i$, where $S_i =1$, if one of the following
three conditions holds:
\begin{align}
\nonumber
&\bullet\,\, A_i =1;\\
\label{consistencycondition}
&\bullet\,\, A_i =\infty,\, D_i =2 \text{ and } Q_i \leq b_i ;\\
\nonumber
&\bullet\,\, A_i =\infty,\, D_i =1 \text{ and } Q_i > b_i .  
\end{align}
If $Q$ is compatible with $b_{i}$ for each $i\in\cS$, we simply say that $Q$ is compatible
with the collection $b=\{b_{i}, i\in\cS\}$, and we write $Q\sim b$.

Now we may rewrite the definition of $w$ around (\ref{weightdef}) as follows. 
Given $A$, $S$ and a valid queue-length process $Q$,
\begin{equation}\label{Wdev}
w(Q|A,S)=\PP_B(Q\sim B).
\end{equation}

\begin{lemma}\label{twoQlemma}
Fix $A$ and $S$.
Suppose that two valid queue-length processes $Q$ and $Q'$ are both compatible with 
the same set of marks $b$. Then 
$Q$ and $Q'$ differ by a constant, 
and also $D(A,S,Q)=D(A,S,Q')$. 
\end{lemma}
\begin{proof}
Let $D=D(A,S,Q)$ and $D'=D(A,S,Q')$
be the departure processes associated to $Q$ and $Q'$ respectively. 
Without loss of generality, suppose that for some $i$, 
$Q'_{i}\geq Q_i $. We wish to show that also $Q'_{i+1}\geq Q_{i+1}$.
Consider two cases. Suppose $Q'_{i}>Q_i $. Since both 
processes have the same arrivals, and there is at most one 
departure at any time, certainly also $Q'_{i+1}\geq Q_{i+1}$.
If instead $Q'_{i}=Q_i $ then from the fact that
$Q$ and $Q'$ are compatible with the same marks $b$,
we have $D'_{i}=D_i $ and hence also $Q'_{i+1}=Q_{i+1}$. 

Hence we obtain $Q'_{i}\geq Q_i $ for all $i$. Again since $Q$ and $Q'$ are compatible with the same marks, we then obtain that $D'$ has a departure whenever $D$ has a departure. 
But the total number of departures
is $k_1$ in each process, so the set of departure times is the same for both. Since the set of arrival times is also the same, the queue-length processes differ by some constant.
\end{proof}

Before stating the next result, we formulate two further equivalent versions of Algorithm \ref{2algo}. 
Recall that Algorithm \ref{2algo} considers the $k_1$ arrivals 
in turn; when considering the $(r+1)$st arrival,
it chooses between the $k_1+k_2-r$ services still available. If a service coincides with 
the arrival, that one is chosen. Otherwise, we number the available services
$i_1, i_2,\dots, i_{k_1+k_2-r}$ in cyclic order, and service $i_j$ is chosen with probability 
$q^{j-1}/(1+q+q^2+\dots+q^{k_1+k_2-r-1})$. 

Equivalently, we can do the following. Now, when an arrival
does not find an available service coinciding with it,
we extend the list of available services 
$i_1,\dots,i_{k_1+k_2-r}$ cyclically
to an infinite sequence
$i_j$, $j\geq 1$, by setting $i_j=i_{j'}$ whenever $j\equiv j' \mod k_1+k_2-r$. 
Now we choose service $i_j$ with probability $q^{j-1}/(1+q+q^2+\dots)=q^{j-1}(1-q)$.
That is, we choose service $i_{J}$ where $J$ is a geometric random variable
with parameter $q$. We have the following interpretation:
we ``offer'' the arrival to each service $i_j$ in turn, proceding 
in cyclic order around the ring. 
Each service accepts the 
arrival with probability $1-q$. If a service accepts the arrival, then the arrival
is assigned to that service; otherwise we continue to the next service in the list.
If we have toured all the way around 
the ring without any offer being accepted, we continue in cyclic order, 
starting again at the beginning. It's easy to check that this procedure
gives the same distribution as the one in the previous paragraph.

Secondly, we can allow a set of marks $b=(b_i, i\in\cS)$ to control the procedure. 
As the algorithm proceeds,
a service at $i$ will ``reject'' the first $b_i $ times it 
receives an offer, and ``accept'' the next time
(if it is ever offered that many).  
If the marks $b_i $ are randomly chosen, with distribution i.i.d.\ geometric with parameter $q$, 
then this is again equivalent to accepting each offer independently 
with probability $1-q$. 

\newcommand{{\Qmin}}{Q_{\min}}

\begin{lemma}\label{Qminlemma}
Fix $A$ and $S$.
For each set of marks $b$, there exists a minimal process $\Qmin(b)$
among those queueing processes compatible with the marks $b$.
\end{lemma}
\begin{proof}
The second reformulation above can also be used to give a process $Q$
which is compatible with a given set of marks $b$. We generate this $Q$ by considering
separately the contribution of each customer in the system. 
Suppose that a given customer arrives at time $i$ and departs at time $j$. This customer contributes 1 to the queue lengths at times $Q_{i+1}, Q_{i+2},\dots, Q_{j-1}, Q_j$. 
In addition, suppose that the match between this customer and the service at $j$ was rejected of $r$ times before being accepted. (In that case, the algorithm has done $j$ ``complete 
tours'' of the ring in the process of assigning the departure time of this customer).
Then this gives a further contribution of $r$ to the queue-length $Q_{i'}$ at all times $i'$.

We claim that summing up these contributions from each customer gives the desired queue-length process.
Each arrival gives an increase of 1 in the queue-length, and each departure gives a decrease of 1.
Further, each time that a service at $i$ is unused by a customer, that customer gives a 
contribution of 1 to the queue-length at $i$; as a result, if the service remains unused at the end of the procedure, we know that $Q_i \leq b_i $. In addition,
any departing customer still contributes at the moment of departure; if a service at $i$ is used for a departure after being rejected $b_i $ times, 
we know that $Q_i \geq b_i +1$. From these properties it follows 
$Q$ is a valid queue-length process for $(A,S)$, and also the $Q$ 
is compatible with the marks $b$. 

Hence there is at least one compatible queueing process. Also Lemma \ref{twoQlemma} states that any two such processes differ by a constant. Finally, all processes are non-negative. So indeed a minimal such process exists. 
\end{proof}


\begin{lemma}\label{conditionallemma}
Let $k_2=|\cS|-\sum \ind(A_i=1)$ be the number of second-class particles. 
Let $Q$ and $Q'$ be valid queue-length processes. 
\begin{equation}
\label{pairprobability}
\PP(Q'\sim B | Q=\Qmin(B))=
\begin{cases}
q^{mk_2} &\text{ if } m\in\{0,1,2,\dots\} \text{ and } Q'_{i}=Q_i +m \text{ for all } i;
\\
0 &\text{ otherwise}.
\end{cases}
\end{equation}
\end{lemma}
\begin{proof}
Suppose $Q=\Qmin(B)$. From Lemma \ref{twoQlemma}, 
if also $Q'\sim B$ then $Q'$ and $Q$ must differ by a constant, say $Q'_{i}=Q_i +m$
for all $i$, and since $Q$ is minimal, we must have $m\geq 0$. 
So it's enough to show that in that case $\PP(Q'\sim B|Q=\Qmin(B))=q^{mk_2}$.

From (\ref{consistencycondition}),
we have that if $Q=\Qmin(B)$, then $B_i\geq Q_i$ for all $i\in\cS$
such that $A_i=\infty$ and $D_i=2$, 
and further that $B\sim Q'$ also holds iff 
$B_i \geq Q_i +m$ for all such $i$.
For each such $i$, we have $\PP(B_i \geq Q_i +m|B_i \geq Q_i )=q^m$,
since $B_i $ is geometric with parameter $q$.  
Since the variables $B_i$ are independent, and
since there are $k_2$ such $i$, the 
overall conditional probability equals $q^{mk_2}$ as required.
\end{proof}

At this point we can deduce the fact that the total weight of all processes $Q$ 
depends on $A$ and $S$ only through the total number $k_1$ of arrivals and $k_2$
of services, which was an important element of the argument in Section \ref{sec:ringproofs}:
\begin{proof}[Proof of Lemma \ref{lem:denominator2}]
\begin{align}
\nonumber
\sum_{Q'} W(Q'|A,S)
&=\sum_{Q'}\PP(Q'\sim B)\\
\nonumber
&=\sum_{Q', Q}\PP(Q=\Qmin(B))\PP(Q'\sim B|Q=\Qmin(B))\\
&=\sum_{Q}\PP(Q=\Qmin(B))\sum_{m=0}^\infty q^{mk_2} \text{\,\,\,\,\,\,(by (\ref{pairprobability}))}
\label{msum}
\\
\nonumber
&=(1-q^{k_2})^{-1}\sum_Q\PP(Q=\Qmin(B))\\
\nonumber
&=(1-q^{k_2})^{-1}.
\end{align}
\end{proof}

\begin{proof}[Proof of Theorem \ref{thm:2algo}]
We wish to show that the distribution on configurations 
obtained from Algorithm \ref{2algo} (or either of the variants 
described before Lemma \ref{Qminlemma}) 
is the same as that given by 
(\ref{weightdef}) and (\ref{2statring}).
Fix $A$ and $S$ and for convenience write $w(Q)=w(Q|A,S)$
and $D(Q)=D(A,S,Q)$.

The distribution at (\ref{weightdef}) and (\ref{2statring}) amounts to the following: choose $Q$ in proportion to the weight $w(Q)$, and then take $D(Q)$. 

As observed at (\ref{Wdev}), $w(Q)=\PP_B(Q\sim B)$. 
We now decompose this weight by writing
$w(Q)=\sum_{Q'}w(Q,Q')$, where we define $w(Q,Q')=\PP(Q\sim B, Q'= \Qmin(B))$.
For all $Q$ and $Q'$ such that this weight is positive,
Lemma \ref{twoQlemma} gives 
that $Q$ and $Q'$ differ by a constant and that $D(Q)=D(Q')$.
Hence equivalent to 
(\ref{weightdef}) and (\ref{2statring})
is the following: choose the pair $Q$ and $Q'$ in proportion 
to the weight $w(Q,Q')$, and then take $D(Q')$.

Next we write $\tw(Q')=\sum_Q w(Q,Q')$. Then we have another equivalent version: choose $Q'$ in proportion to the weight $\tw(Q')$,
and take $D(Q')$. 

Now
\begin{align}
\nonumber
\tw(Q')
&=\sum_Q \PP_B(Q\sim B, Q'=\Qmin(B))\\
\nonumber
&=\PP_B(Q'=\Qmin(B))\sum_Q \PP_B(Q\sim B|Q'=\Qmin(B))\\
\nonumber
&=\PP_B(Q'=\Qmin(B))(1+q^{k_2}+q^{2k_2}+\dots)
\text{\,\,\,\,\,\,(using Lemma \ref{pairprobability})}\\
\label{constant}
&=\frac1{1-q^{k_2}}\PP_B(Q'=\Qmin(B)).
\end{align}

Since $1/(1-q^{k_2})$ is a constant, we have another equivalent version: choose $Q'$ with probability proportional to
the quantity $\PP_B(Q'=\Qmin(B))$ and take $D(Q')$.
But of course this is the same as the following;
generate a sample of the weights $B$, and take 
$D(\Qmin(B))$. 

Finally, using Lemma \ref{twoQlemma} again,
we may do the following:
generate $B$, choose \textit{any} $Q$ such that $Q\sim B$, and take the departure process $D(Q)$. But, by the argument preceding Lemma 
\ref{Qminlemma} above, this is equivalent to what Algorithm \ref{2algo} does. 

Hence indeed Algorithm \ref{2algo} leads to the same distribution as 
(\ref{weightdef}) and (\ref{2statring}), as desired.
\end{proof}

This completes the proof of the two-type result
in Theorem \ref{thm:2algo}. The structure of the 
proof of the multi-type result in Theorem \ref{thm:multialgo} is entirely analogous. We outline the generalisation of the argument.

As before, we consider marks $B_i $ at each of the service times $i$,
which are i.i.d.\ geometric with parameter $q$. The mark $B_i $
represents the number of times that the service at time $i$ may 
be rejected. A multi-type process $Q$ is compatible with a set of marks $b$
(for which we write $Q\sim B$) 
if each of its embedded one-type queues is compatible with $b$. 
That is,
for each $n$ and for each $i$ such that $S_i =1$, one of the following conditions 
holds, just as at (\ref{consistencycondition}):
\begin{align}
\nonumber
&\bullet\,\, A^{(\leq n)}(i)=1;\\
\label{multiconsistencycondition}
&\bullet\,\, A^{(\leq n)}(i)=0,\, D^{(\leq n)}(i)=0 \text{ and } B_i \geq Q^{(\leq n)}(i);\\
\nonumber
&\bullet\,\, A^{(\leq n)}(i)=0,\, D^{(\leq n)}(i)=1 \text{ and } B_i <Q^{(\leq n)}(i).  
\end{align}

Just as at (\ref{Wdev}), we again have
\begin{equation}
W(Q|A,S)=\PP_B(Q\sim B).
\end{equation}
and analogously to Lemma \ref{twoQlemma}, we have
\begin{lemma}
Suppose that two queue-length processes $Q$ and $Q'$ are both compatible with 
the same set of marks $b$. Then 
$Q$ and $Q'$ differ by a constant vector, 
in the sense that for some $m_1, m_2, \dots m_{N-1}$,
one has $Q'^{(\leq n)}(i)=Q^{(\leq n)}(i)+m_n$ for all $n$ and $i$. 
In particular, all the departure 
processes $D^{(\leq n)}$ are identical 
in $Q$ and in $Q'$. 
\end{lemma}

The equivalent of Lemma \ref{Qminlemma} also holds, 
so that for any set of marks $b$ there 
exist queueing processes compatible with $b$ and among them a
minimal process
$\Qmin(b)$
among those queueing processes compatible with the marks $b$.
Here the process $\Qmin$ is minimal in the sense
that for any other compatible $Q$ and any $n$ and $i$,
$\Qmin^{(\leq n)}(i)\leq Q^{(\leq n)}(i)$. 
(It may not necessarily be the case that $\Qmin^{(n)}(i)\leq Q^{(n)}(i)$).

Indeed, the multi-type $\Qmin$ can be obtained by taking
each $\Qmin^{(\leq n)}$ to be 
the minimal process for the embdedded one-type queue, as given by 
Lemma \ref{Qminlemma}.

The following generalisation of Lemma \ref{conditionallemma}
then holds:
\begin{lemma}\label{multiconditionallemma}
Let $A$ be an arrival process with particle counts 
$k_1,\dots, k_{N-1}$. Let $S$ be a service process
with $k_1+\dots+k_N$ services. 
Let $Q$ and $Q'$ be two valid queue-length processes. 
If there exist $m_n\geq 0$, $1\leq n\leq N-1$ such 
that for all $n$ and $i$,
\[
Q'^{(\leq n)}(i)=Q^{(\leq n)}(i)+m_n
\]
then
\[
\PP_B(Q'\sim B|Q=\Qmin(B))=q^{m_1k_2+m_2k_3+\dots+m_{N-1}k_N}.
\]
Otherwise $\PP_B(Q'\sim B|Q=\Qmin(B))=0$. 
\end{lemma}

At this point we have a proof of Lemma \ref{lem:denominatorN}:

\begin{proof}[Proof of Lemma \ref{lem:denominatorN}]
Lemma \ref{lem:denominatorN} follows from Lemma \ref{multiconditionallemma}
in just the same way that Lemma \ref{lem:denominator2} followed from Lemma
\ref{conditionallemma}. In place of the sum over $m$ in (\ref{msum}), 
we will have a sum over $m_1, m_2, \dots, m_{N-1}$.
\end{proof}

To complete the proof of Theorem \ref{thm:multialgo},  
start from $w(Q|A,S)=\PP_B(Q\sim B)$.
Define again $w(Q,Q')=\PP_B(Q\sim B, Q'=\Qmin(B))$
and then $\tW(Q')=\sum_Q w(Q,Q')$. The
equivalent step to (\ref{constant}) is that
\begin{align*}
\tw(Q')
&=\sum_Q \PP_B(Q\sim B, Q'=\Qmin(B))\\
&=\PP_B(Q'=\Qmin(B))\sum_Q \PP_B(Q\sim B|Q'=\Qmin(B))\\
&=\PP_B(Q'=\Qmin(B))\sum_{m_1,m_2,\dots, m_{N-1}\geq0}
q^{m_1k_2+m_2k_3+\dots+m_{N-1}k_N}\\
&=\frac1{(1-q^{k_2})(1-q^{k_3})\dots(1-q^{k_N})}\PP_B(Q'=\Qmin(B)).
\end{align*}
The denominator in the final expression is constant,
and the end of the proof goes through just as before.

\subsection{Proof of the results on common denominators}
\label{sec:denominatorsproof}
As noted in Section \ref{sec:intromain},
we can generate a sample from the $N$-type ASEP on $\ZZ_L$
by applying Algorithm \ref{multialgo} $N-1$ times in all. 
The $n$th iteration takes as arrival process a configuration
whose distribution is stationary for an $n$-type ASEP,
along with an independent service process, and outputs a 
configuration whose distribution is stationary for an 
$(n+1)$-type ASEP. The $N-1$ iterations can be combined
into a ``multi-line" diagram with $N$ lines. In particular,
the bottom line of the diagram gives a sample from an $N$-type
system.

We can identify two sources of randomness within this procedure;
the choice of the configuration of occupied and unoccupied sites on the lines of the multi-line diagram, and the assignment of types to
the occupied sites. The set of occupied sites on line $n$ 
is chosen uniformly from all subsets of $\ZZ_L$ of size 
$k_1+\dots+k_n$, independently for different lines; hence overall
the configuration is uniformly chosen from 
\begin{equation}\label{sitechoices}
\Ch{L}{k_1}
\Ch{L}{k_1+k_2}
\dots
\Ch{L}{k_1+k_2+\dots+k_N}.
\end{equation}
possibilities. 

To assign types to particles in line $r$, given the types in line $r-1$,
we use Algorithm \ref{multialgo}. 
We treat the arrivals in 
order of their type
(within a single type, the order in which we treat
the arrivals is irrelevant -- for example, 
we can work from left to right).
If $m$ of the arrivals (from line $r-1$)
have been assigned to service-times (on line $r$), then there remain
$k_1+\dots+k_{r}-m$ services. When we come to assign the $(m+1)$st arrival,
there are two cases. Either there is an unassigned service immediately below it, in which it is assigned there. Otherwise, the choice between the remaining services is done according to a distribution with weights
\begin{equation}
\label{typechoices}
\frac{1}{1+q+\dots+q^{k_1+\dots+k_{r}-m-1}}
\left(
1,q,\dots, q^{k_1+\dots+k_{r}-m-1}
\right).
\end{equation}

Once the set of occupied sites is fixed, 
the probability of obtaining
any particular multiline diagram is given by a product
of terms of the form (\ref{typechoices}),
where $r,m$ vary over 
$1\leq m\leq k_1+k_2+\dots+k_{r-1}$.
To obtain the polynomial (\ref{generaldenom})
we take the product over such $r,m$ of the denominator $[k_1+\dots+k_r-m-1]_q$ from
(\ref{typechoices}), and multiply by (\ref{sitechoices}).
Every numerator term in (\ref{typechoices}) is a
power of $q$, so indeed the probability of any multiline 
diagram is a polynomial in $q$ with non-negative
integer coefficients divided by the polynomial 
(\ref{generaldenom}). But the bottom line 
of the multiline diagram gives a sample
from the stationary distribution of the $N$-type ASEP. 
So the stationary probabilities are sums of 
probabilities of multiline diagrams, and so themselves
have the same form, giving Theorem \ref{thm:denominators}.

For Corollary \ref{cor:Ldenominators}, we consider
an $(L-1)$-type system with $k_r=1$ for $r=1,2,\dots, L-1$.
(The single hole plays the role of the particle of
type $L$.) In this case (\ref{generaldenom})
reduces to (\ref{Ltypedenom}). 

\section{Systems on the line}
\label{sec:Z}
Now we want to consider systems in which the set of sites
is given by the whole integer line $\ZZ$ rather than the ring
$\ZZ_L$. 

Given $\lambda_1, \dots, \lambda_N$ with 
$\sum \lambda_i<1$, there is a unique
ergodic translation-invariant stationary distribution
where the density of particles of type $i$ is $\lambda_i$,
as shown by Ferrari, Kipnis and Saada
\cite{FerrariKipnisSaada}.

We want to prove Theorem \ref{thm:2statline} for the
2-type case and Theorem \ref{thm:Nstatline} for the multi-type case. In each case we will prove stationarity for the measure on the whole line by showing that it is the limit of the
stationary distribution of systems on finite rings. 

A probability distribution $\nu$ on 
$\{1,2,\dots, N, \infty\}^\ZZ$ is characterised 
by its values on cylinder events, i.e.\ those which depend
on some finite window $\{-m, \dots, m\}\subset\ZZ$. 

For $T\in\NN$, consider systems on the ring of size $L=2T$,
with the sites labelled as $-T, \dots, T-1$. 
Let $\nu^{2T}$ be any sequence of stationary distributions 
for these systems. If $C$ is a cylinder event which depends on a set of sites $-m, \dots, m$, then we can consider
the probability of $C$ under $\nu^{2T}$ for any $T>m$. 

\begin{proposition}\label{prop:cylinder}
If the distribution $\nu$ satisfies $\nu^{2T}(C)\to\nu(C)$
as $T\to\infty$ for every cylinder event $C$, then 
$\nu$ is a stationary distribution for the system on $\ZZ$.
\end{proposition}
\begin{proof}
Let $C$ depend on the sites $-m, \dots, m$ 
and fix any $t>0$. Consider a system on the ring 
of size $2T$ started from some state 
$\eta_0^{(T)}\sim \nu^{2T}$ at time $0$, and a system on $\ZZ$ 
started from some state $\eta_0\sim \nu$ at time $0$.

Since $\nu^{2T}\to\nu$ on cylinder events, for any given $M$,
whenever $T$ is large enough,
we can couple $\eta_0^{(T)}$ and $\eta_0$ so that
the probability that they are identical on the window 
$\{-M, -M+1, \dots, M\}$ is as close to $1$ as desired. 

Meanwhile, if we choose $M$ large enough, we can couple
the evolutions on the finite time interval $[0,t]$
in such a way that if the time-$0$ states are identical
on the window $\{-M,\dots,M\}$ then (uniformly in the time-$0$ state)
the time-$t$ states are identical on the smaller 
window $\{-m, \dots, m\}$ with probability as close to $1$
as desired. (This can be achieved by a simple local coupling 
of the dynamics. We omit the details but a straightforward approach is as follows: the process of jumps between sites
$i$ and $i+1$ can be dominated by independent Poisson 
processes $P_x$ of rate $1+q$. For any given $i$, the probability that $P_i$ has no point in the time interval $[0,t]$ is
$e^{-t(1+q)}$; hence with high probability as $M\to\infty$, we can find such sites $i_0\in\{-M, \dots, -m-1\}$ and 
$i_1\in\{m, \dots, M-1\}$. Then one can couple 
such that in both systems, no particle enters or leaves the
interval $\{i_0+1, \dots, i_1\}$ during $[0,t]$, and such 
that inside that interval, the evolutions of the two systems
are identical. In particular, the two evolutions coincide
on the window $\{-m, \dots, m\}$.)

Hence, restricted to $-m,\dots,m$, 
the distribution of the time-$t$ state in the infinite system
is arbitrarily close to that of the time-$t$ of the system on the ring. 
But on the ring, the distribution of the time-$t$ state is 
$\nu^{2T}$, since $\nu^{2T}$ is stationary. Also by choosing large enough $T$, $\nu^{2T}$ 
can be made to agree arbitrarily closely with $\nu$ on the window $-m,\dots,m$. Hence the probability of the event $C$ occurring at time $t$ in the system on $\ZZ$ must be $\nu(C)$. This holds for all cylinder events $C$, so 
indeed the distribution at time $t$ is $\nu$; hence $\nu$ is a 
stationary distribution as required.
\end{proof}

Now we turn to the queueing Markov chain. We will
give it a rather explicit construction, involving data
like the ``marks" that we used in the previous section
when justifying Algorithms \ref{2algo} and \ref{multialgo}.

We will start by giving complete details in the case $N=2$, 
where the chain consists of a single queue. Then we will indicate the extension to larger $N$ (where the chain consists of several queues in tandem, and so the details are somewhat more complicated to write down in full, although entirely analogous).

\subsection{$N=2$ (single-type queue)}
First we introduce a rather concrete representation
of the Markov chain used in Theorem \ref{thm:2statline}.
We consider random vectors $R_i =(A_i , S_i , B_i ), i\in\ZZ$,
which control the evolution of the queue. All entries are independent, with
\begin{gather*}
\PP(A_i=1)=1-\PP(A_i=\infty)=\lambda,\\
\PP(S_i=1)=1-\PP(S_i=\infty)=\mu,\\
B_i \text{ geometric with parameter } q.
\end{gather*}
Here $A_i=1$ means that there is an arrival at $i$.
$S_i=\infty$ means there is no service available at $i$.
If on the other hand $S_i=1$, then there is a potential 
service available at $i$, which is used if there are at least
$B_i$ customers in the queue, or if an arrival has just occurred; otherwise it is unused. 

Writing $Q_i$ for the queue-length at the beginning of 
time-step $i$, we can write a recursion of the form
\begin{equation}
\label{fdef}
Q_{i+1}=f(Q_i , R_i )
\end{equation}
for some appropriate function $f$ (which would
be easy to write down, but the precise form is not important). This formal representation is useful because it allows us to 
couple versions of the system evolving 
from different initial conditions but using the 
``same randomness" $(R_i , i\in\ZZ)$,
and, particularly, to couple versions of the system evolving
on the whole real line $\ZZ$ or on a finite box $[-T,T]$.

Similarly, we can write the departure process defined at (\ref{2stat}) in the form
\begin{equation}
\label{gdef}
D_i=g(Q_i, R_i),
\end{equation}
again for some suitably chosen function $g$. 

\subsubsection{Cyclic evolution}
Let $G_T$ be the event that 
there are more services than arrivals 
in the finite box; specifically, that 
$\sum_{i=-T}^{T-1} \ind(A(i=1)) < \sum_{i=-T}^{T-1} \ind(S_i =1)$.
We have 
$\sum_{i=-T}^{T-1} \ind(A(i=1))\sim \Bin(2T, \lambda)$
and
$\sum_{i=-T}^{T-1} \ind(S_i =1)\sim \Bin(2T, \mu)$,
with $\lambda<\mu$, so 
\begin{equation}\label{GToccurs}
\PP(G_T)\to 1 \text{ as }T\to\infty.
\end{equation}

\newcommand{\cyc}{\textrm{cyc}}
\newcommand{\Qc}{Q^{\cyc}}
\newcommand{\Dc}{D^{\cyc}}

When $G_T$ holds, we saw in the proof of \ref{Qminlemma}
how to construct a \textit{cyclic evolution} compatible with the randomness;
that is, an evolution $\Qc_{-T}, \dots, \Qc_T$ with the property 
that $\Qc_{-T}=\Qc_T$, and such that for all $i=-T, \dots, T-1$, 
$\Qc_{i+1}=f(\Qc_i, R_i)$ as in
(\ref{fdef}) holds. In Lemma \ref{twoQlemma} we saw that any two such cyclic evolutions differ by a constant, and that furthermore 
any two such evolutions give the same output configuration
$\Dc=(\Dc_i , i\in[-T, T-1])$ (where $\Dc_i=g(\Qc_i, R_i)$ as at
(\ref{gdef})).

Conditional on the number of arrivals and services, the distribution of this configuration is stationary for the ASEP on $\ZZ_{2T}$. Let $\nu^{2T}_{\lambda, \mu-\lambda}$ be the distribution of $\Dc$, conditioned
on the event $G_T$. Then $\nu^{2T}_{\lambda, \mu-\lambda}$ is a mixture
of stationary distributions for the ASEP
(by Theorem \ref{thm:2statring}), and so is
itself stationary. 

\newcommand{\Qeq}{Q^{\textrm{eq}}}
\newcommand{\Deq}{D^{\textrm{eq}}}

\subsubsection{Coupling of cyclic evolution to an evolution on $\ZZ$}
\label{subsec:coupling}
The queue-length Markov chain 
described in Section \ref{sec:basicqueue}
has a stationary distribution
$\pi$. (This distribution is easy to obtain using
the detailed balance equations, but we don't need its
particular form here.)
Suppose $\Qeq_{-T}$ is drawn from the stationary distribution $\pi$ 
of the queue-length Markov chain, 
independently of the randomness $(B_i , A_i , S_i ), 
i=-T, \dots, T-1$.
Then define $\Qeq_i , i=-T+1, \dots, T$ 
by $\Qeq_{i+1}=f(\Qeq_i, R_i)$ as at $(\ref{fdef})$.
We call this the equilibrium evolution. 
Similarly define 
$\Deq_i, i=-T+1, \dots, T$ as at $(\ref{gdef})$. 
Then $(\Deq_i, i\in[-T, T-1])$
is distributed according to the restriction of 
$\nu^{(2)}_{\lambda,\mu-\lambda}$ to the interval
$[-T, T-1]$.

Next, define $\tQ(-T)=\Qeq(T)$, and then $\tQ_{i+1}=
f(\tQ_i , R_i )$ 
for $i=-T, \dots, T-1$. So $\tQ$ also evolves as a copy
of the queue-length chain on $[-T,T]$, using the 
same randomness as $\Qeq$, but started from the particular
initial state $\Qeq_T$. We claim that with high probability,
$\Qeq$ and $\tQ$ couple together on the interval; that 
is, they reach the same state at some point, and then
since they use the same randomness, they stay together 
for the rest of the interval also. Then $\tQ_T=\Qeq_T=\tQ_{-T}$,
and so in particular $\tQ$ is a cyclic evolution. In fact,
we aim to show that with high probability, the coupling point occurs before $-m$, so that also $\tQ_i =\Qeq_i $ for all 
$i\in[-m,m]$. See Figure \ref{fig:coupling}
for an illustration of the idea.

\begin{figure}[htb]
\resizebox{\textwidth}{!}{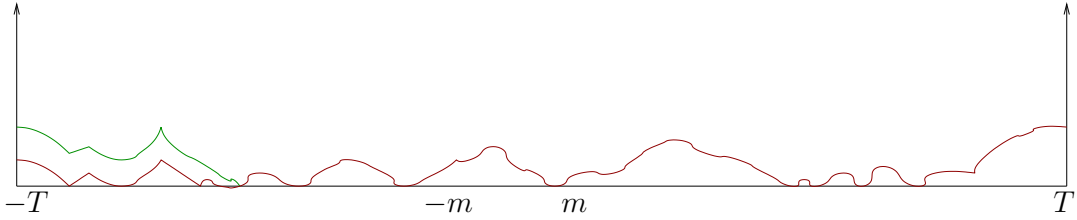}
\caption{
\label{fig:coupling}
Illustration of the coupling idea in Section \ref{subsec:coupling}. 
The red curve (which starts lower) illustrates 
the Markov chain evolving in equilibrium starting from the
state $\Qeq_{-T}$, and ending at the state $Q_{T}$.
The green curve starts at the state $\tQ_{-T}:=Q_{T}$, and 
evolves using the same randomness $R_i, -T\leq i<T$ as the red curve. 
If the two evolutions meet, they stay together. In this case the 
evolution starting from $\tQ_{-T}$ is cyclic. If the coupling point is before time $-m$, then the cyclic evolution and the equilibrium evolution are identical on the window $[-m,m]$. 
We show that (for fixed $m$) this happens
with high probability as $T$ becomes large, so that the distribution 
on $[-m,m]$ for the cyclic evolution on $[-T,T]$ 
converges to that for the equilibrium evolution.}
\end{figure}

\begin{proposition}
\label{prop:coupling}
As $T\to\infty$,
\[
\PP(\Qeq_t=\tQ_t \text{ for some } t \text{ with } 
-T<t<-m) \to 1.
\]
If this event holds then $\Qeq_i=\tQ_i$ for all $t\leq i\leq T$, 
and in particular:
\begin{itemize}
\item[(i)]
$\tQ_T=\Qeq_T=\tQ_{-T}$, so that $\tQ$ is a cyclic evolution;
\item[(ii)]
$\tQ_i=\Qeq_i$ for $i\in[-m,m]$.
\end{itemize} 
\end{proposition}

If both $G_T$ and the event in (ii) occur, then 
the equilibrium departure process $\Deq$ and the 
cyclic departure process $\Dc$ agree on the window 
$[-m,m]$. Since $m$ is arbitrary, we obtain from Proposition 
\ref{prop:coupling} together with (\ref{GToccurs}):
\begin{corollary}\label{cor:cylinder-convergence}
For any cylinder event $C$, $\nu_{\lambda, \mu-\lambda}^{2T}(C)\to 
\nu^{(2)}_{\lambda, \mu-\lambda}(C)$ 
as $T\to\infty$. 
Hence by Proposition \ref{prop:cylinder},
$\nu$ is a stationary distribution for the system on $\ZZ$ as required. 
\end{corollary}

In the rest of this section we prove Proposition \ref{prop:coupling}. We want to consider a ``pair chain",
which consists of two copies of the chain evolving
with the same randomness.

\begin{lemma}\label{lem:pair-coupling}
Fix any $x_0$ and ${x_0}'$. Define
$Q_0=x_0$, $Q'_0=x'_0$, and 
\begin{equation}\label{pair-updates}
(Q_{t+1}, Q'_{t+1})=(f(Q_t, R_t), f(Q'_t, R_t))
\end{equation}
for $t\geq 0$.

If $\epsilon>0$, then there exists $t_0$ such that
\[
\PP(Q_t=Q'_t\text{ for all } t\geq t_0)>1-\epsilon.
\]
\end{lemma}

\begin{proof}
The underlying Markov chain, with updates given by (\ref{fdef}), is irreducible, aperiodic, and positive recurrent, with a unique stationary distribution $\pi$.
Fix some $k$ such that $\sum_{x=0}^{k} \pi(x)>3/4$. 
Starting from any initial state, we can 
apply the Markov chain ergodic theorem to deduce
that with probability 1, the long-run proportion of 
time which the chain spends in $[0,k]$ is more than $3/4$. 

Now consider the pair chain with updates given by 
(\ref{pair-updates}). Starting from $(x_0, x'_0)$, with probability one, eventually each coordinate of the chain spends at least $75\%$ of its time in $[0,k]$. Hence the pair-chain spends at least half its time in the set $\cH_k:=[0,k]\times[0,k]$. In particular, it visits that set infinitely often. 

Consider the pair chain evolving from any state in $\cH_k$. With some uniformly positive probability, after $k$ more steps the chain is in the state $[0,0]$. For example, it suffices that at each of the next $k$ steps, there is no arrival, and a service which is accepted by the first customer in the queue (if any is present). This happens with probability at least $[(1-\lambda)\mu(1-q)]^k$. 

Since the chain visits $\cH_k$ infinitely often with probability 1, and from any state in $\cH_k$ there is uniformly positive probability to reach the state $(0,0)$ within $k$ steps, it follows that with probability 1 the chain 
will eventually visit state $(0,0)$. So certainly there is 
some time $t_0$ such that with probability $1-\epsilon$,
the chain visits $(0,0)$ before time $t_0$. But once
it visits $(0,0)$, the two coordinates of the chain 
stay the same for ever. 
\end{proof}

\begin{proof}[Proof of Proposition \ref{prop:coupling}]
In the setting of Proposition \ref{prop:coupling},
we start the pair-chain from a particular state 
$(\Qeq_{-T}, \tQ_{-T})$.
We wish to show that for fixed $m$, the probability 
that the two coordinates of the chain couple 
before time $-m$, i.e.\ within $T-m$ steps,
tends to $1$ as $T\to\infty$.

Note that Lemma \ref{lem:pair-coupling} tells us
that started from any fixed state, there is a $t_0$
sufficiently large that with probability at least 
$1-\epsilon$, the two coordinates of the chain couple
within $t_0$ steps. It remains to deal with the fact that
the initial condition $(\Qeq_{-T}, \tQ_{-T})$ is random rather
than fixed. 

Note that $(\Qeq_t, t\in\ZZ)$ is an equilibrium version of the chain. In particular, both $\Qeq_{-T}$ and $\tQ_{-T}=\Qeq_{T}$ have distribution $\pi$. For any given $\delta$, choose $k$ sufficiently large
that $\sum_{x=k+1}^\infty \pi_x\leq\delta/3$. 
Then 
\begin{align*}
\PP((\Qeq_{-T}, \tQ_{-T})\notin [0,k]\times[0,k])&
\leq \PP(\Qeq_{-T}\notin[0,k])+\PP(\Qeq_{T}\notin[0,k])\\
&=2\sum_{x=k+1}^\infty \pi_x
\\
&\leq 2\delta/3.
\end{align*}

Now using Lemma \ref{lem:pair-coupling}, there is some 
$t_0$ such that, started from any given state in $[0,k]\times[0,k]$, the probability of failing to couple within $t_0$ steps is 
at most $\delta/3(k+1)^2$. Then since there are $(k+1)^2$ possible initial states in $[0,k]\times[0,k]$,
the probability that there exists at least one initial state in $[0,k]\times[0,k]$ such that failure to couple occurs (using the updates
(\ref{pair-updates})) is at most $\delta/3$.

Then for all $T$ large enough that $T-m>t_0$, indeed
with probability at least $1-\delta$, the pair-chain
started from state $(\Qeq_{-T}, \tQ_{-T})$ at time $-T$, 
and updated using (\ref{pair-updates}), couples before
time $-m$. Since $\delta$ is arbitrary, this
gives Proposition \ref{prop:coupling} as required. 
\end{proof}

Hence we have Corollary \ref{cor:cylinder-convergence},
and via Proposition \ref{prop:cylinder}
we have justified the construction of the $2$-type stationary distribution on $\ZZ$ in Theorem \ref{thm:2statline}.

\subsection{$N>2$ (multi-type queue)}
Everything in the previous section generalises
straightforwardly to the case $N>2$,
to prove Theorem \ref{thm:Nstatline} constructing
the stationary measure of the $N$-type system on $\ZZ$.
Let us outline the changes. 

The queueing Markov chain is now a system of queues in 
tandem. There are $N-1$ queues, and the departure process
of queue $r$ is the arrival process for queue $r+1$,
for $1\leq r\leq N-2$. (Note that a customer leaving 
queue $r$ at time $i$ arrives at queue $r+1$ in the 
same time-slot $i$, so may pass through queue $r+1$,
and further queues, in the same time-slot.)

The state of the system at time-step $i$ may be 
written as 
\[
\bQ_i=\left(Q_i^{(n,r)}, 1\leq n\leq r\leq N-1\right),
\]
where $Q_i^{(n,r)}$ is the number of customers
of type $n$ in queue $r$ at the beginning of time-step $i$.
The evolution of the queues can again be controlled by a vector of random variables corresponding to each time-step. 
Now we would have 
$R_i=\left(A_i, (S^{(r)}_i, 1\leq r\leq N-1), 
(B^{(r)}_i, 1\leq r\leq N-1)\right)$, where again all entries
are independent, with 
\begin{gather*}
\PP(A_i=1)=1-\PP(A_i=\infty)=\lambda_1,\\
\PP\left(S_i^{(r)}=1\right)
=1-\PP\left(S_i^{(r)}=\infty\right)=\mu_r:=
\sum_{n=1}^{r+1} \lambda_n,\\
B_i^{(r)} \text{ geometric with parameter } q.
\end{gather*}
The $(S_i^{(r)})$ provide the service process of queue $r$,
and the $(B_i^{(r)})$ govern the random choices of acceptance and rejection of service at queue $r$. The dynamics
of the queues are then just as we developed in  
Section \ref{sec:queueing}. 
Just as at (\ref{fdef}), we can 
construct the evolution of the Markov chain in the form 
\[
\bQ_{i+1}=f\left(\bQ_i, R_i\right)
\]
for some appropriate function $f$, 
and this again allows us to couple versions
of the system starting from different initial conditions. 

The equivalent of the event $G_T$ defined at 
(\ref{GToccurs}) is now that, on the time interval
$[-T, T-1]$,
queue $1$ has more services than arrivals than services,
and queue $r+1$ has more services than queue $r$ does
for each $r=1,2,\dots, N-2$. Again the probability of this
event tends to $1$ as $T\to\infty$. 

The Markov chain is now multi-dimensional, and its
stationary distribution is no longer easy to obtain. 
However, it is still positive recurrent
(as can be seen by considering the marginal evolution
of each queue individually; for any given $r$, the process of
the total number of customers in the $r$th queue,
namely $\sum_{n=1}^r Q_i^{(n,r)}, i\in\ZZ$, 
behaves like a single-type queue with 
arrivals at rate $\sum_{n=1}^r \lambda_n$
and potential services at rate $\sum_{n=1}^{r+1} \lambda_n$).

Then exactly the same ideas for coupling the cyclic evolution and the equilibrium evolution apply, as illustrated in Figure \ref{fig:coupling}.

For later reference, we summarise the key points 
in the following result:
\begin{proposition}
\label{prop:linesummary}
Fix $\lambda_1, \dots, \lambda_N\in(0,1)$ 
with $\lambda_1+\dots+\lambda_N<1$. 

Let $M_n\sim\Bin(2T, \lambda_n)$ independently for $n=1,2,\dots,N$.
Let $G_T$ be the event $M_1<M_2<\dots<M_N$. 
Then $\PP(G_T)\to1$ as $T\to\infty$. 

Define $K_1=M_1$ and $K_n=M_n-M_{n-1}$ for $n=2,\dots,N$.

Let $\nu^{2T}_{\lambda_1,\dots,\lambda_N}$
be the distribution of a configuration drawn from the stationary
distribution of the ASEP on $[-T,T-1]$ with $K_n$ particles of 
type $n$ for $n=1,\dots,N$, conditional on $G_T$. 

Then $\nu^{2T}_{\lambda_1,\dots,\lambda_N}$
is a mixture of stationary distributions for the $N$-type ASEP
on $[-T,T-1]$ and so is itself stationary. 

As $T\to\infty$,
$\nu^{2T}_{\lambda_1,\dots,\lambda_N}(C)\to\nu^{(N)}_{\lambda_1,\dots,\lambda_N}(C)$
for all cylinder events $C$,
where $\nu^{(N)}_{\lambda_1,\dots,\lambda_N}$
is the distribution defined in Theorem \ref{thm:Nstatline}. 
\end{proposition}
From this cylinder convergence, using Proposition \ref{prop:cylinder},
we get the result of Theorem \ref{thm:Nstatline} as required.

\section{Clustering}
\label{sec:convoys}
In this section we prove Theorem \ref{thm:Wlimit}.
Recall that
$Y^{(L)}$ denotes a configuration drawn from the
stationary distribution of an ASEP with $L$ types on the ring
with $L$ sites 
$-\lfloor\frac L2\rfloor, \dots,$ $-1,$ $0,$ $1,$ $\dots, \lceil\frac L2\rceil-1$, and as at (\ref{WLdef}), we put
\begin{equation*}
W^{(L)}=
\left(
\dots,0,0,0,\tfrac1L Y^{(L)}_{-\lfloor\frac L2\rfloor},\dots, \tfrac1L Y^{(L)}_{-1},
\tfrac1L Y^{(L)}_{0}, \tfrac1L Y^{(L)}_{1}, \dots, \tfrac1L Y^{(L)}_{\lceil\frac L2\rceil-1},0,0,0,
\dots\right).
\end{equation*}
We want to show that $W^{(L)}$ converges in distribution as 
$L\to\infty$. 
Some of the arguments we need involving convergence
of stationary distributions on the ring to those on the line
have already been developed in the previous section during the proof
of Theorems \ref{thm:2statline} and \ref{thm:Nstatline}. 
We begin with a variation of Proposition \ref{prop:linesummary}.

\begin{proposition}
\label{prop:locallimit}
Fix $\lambda_1, \dots, \lambda_N\in(0,1)$ 
with $\lambda_1+\dots+\lambda_N<1$. 

Suppose the sequences $k_1^{(L)}, \dots, k_N^{(L)}$ 
of positive integers satisfy $k_n^{(L)}/L \to \lambda_n$
as $L\to\infty$, for $n=1,\dots, N$.

Let $\nu^{L}_{k_1^{(L)},\dots, k_N^{(L)}}$ be the 
stationary distribution of the ASEP on a ring of $L$
sites, labelled $-\lfloor\frac L2\rfloor, \dots,$ $-1,$ $0,$ $1,$ $\dots, \lceil\frac L2\rceil-1$ in a cyclic way, 
with $k_n^{(L)}$ particles of type $n$ for $n=1,\dots,N$. 

Let $\nu_{\lambda_1,\dots, \lambda_N}$
be the ergodic translation-invariant stationary 
distribution of the $N$-type ASEP on $\ZZ$ with density 
$\lambda_n$ of particles of type $n$ for $n=1,\dots,N$,
given by Theorem \ref{thm:Nstatline}.

Then $\nu^{L}_{k_1^{(L)},\dots, k_N^{(L)}}(C)
\to\nu_{\lambda_1,\dots, \lambda_N}(C)$ as $L\to\infty$ 
for all cylinder events $C$.
\end{proposition}

\begin{proof}
We indicate the differences between this
result and Proposition \ref{prop:linesummary} above.

A first and rather trivial difference is that we no 
longer assume $L$ is even. In the previous section we took
$L=2T$ throughout, but this was purely for notational 
convenience and makes no difference to the method of proof. 

The more substantial difference is that now we take 
the number of particles of each type to be deterministic,
rather than given in terms of binomial random variables as
in Proposition \ref{prop:linesummary}. 

We may couple configurations distributed according to $\nu_{\lambda_1-\epsilon,\lambda_2, \dots, \lambda_N}$
and $\nu_{\lambda_1,\lambda_2, \dots, \lambda_N}$ so that on any
given set of $m$ sites, they disagree with probability at most 
$N m \epsilon$. (This can be done for example by constructing
both measures as different projections of the $2N$-type measure
with densities $\lambda_1-\epsilon, \epsilon, \lambda_2-\epsilon, \epsilon, \dots, \lambda_N-\epsilon, \epsilon$.)
The same holds for the measures 
$\nu_{\lambda_1,\lambda_2, \dots, \lambda_N}$
and $\nu_{\lambda_1+\epsilon,\lambda_2, \dots, \lambda_N}$.
	
Via similar couplings, and using the law of large numbers
for binomial random variables as sums of independent Bernoulli 
random variables, we can couple configurations
$\eta^L_-\sim \nu^L_{\lambda_1-\epsilon,\lambda_2, \dots, \lambda_N}$,
$\eta^L \sim \nu^L_{\lambda_1,\lambda_2, \dots, \lambda_N}$, and
$\eta^L_+\sim \nu^L_{\lambda_1+\epsilon,\lambda_2, \dots, \lambda_N}$
so that with probability tending to $1$ as $L\to\infty$,
$\eta^L_+\leq \eta^L \leq \eta^L_-$ componentwise.

Combining these two facts with the convergence 
(in terms of probabilities of cylinder events)
of 
$\nu^L_{\lambda_1-\epsilon,\lambda_2, \dots, \lambda_N}$
and 
$\nu^L_{k_1, k_2, \dots, k_N}$
to 
$\nu_{\lambda_1-\epsilon,\lambda_2, \dots, \lambda_N}$
and 
$\nu_{\lambda_1+\epsilon,\lambda_2, \dots, \lambda_N}$
respectively, as given by Proposition \ref{prop:linesummary},
and taking $\epsilon$ to $0$,
we then also get the convergence of the probability of any 
cylinder event $C$ under 
$\nu^L_{k_1, k_2, \dots, k_N}$
to that under $\nu_{\lambda_1,\lambda_2, \dots, \lambda_N}$
as $L\to\infty$, as required for Proposition \ref{prop:locallimit}.
\end{proof}

Now consider the probabilities of cylinder events written in terms
of $W^{(L)}$. In particular, appropriate projections
of $W^{(L)}$ give us systems with a finite number of types,
to which we can apply Proposition \ref{prop:locallimit}.

\begin{proposition}
Let $i\in\ZZ$ and let $x,y\in(0,1)$. As $L\to\infty$,
\begin{equation}\label{projection2}
\PP\left(W^{(L)}_i\leq x, W^{(L)}_{i+1}\leq y\right)
\to
\begin{cases}
\nu_{x,y-x}\left(\eta_0=1, \eta_1\in\{1,2\}\right)&\text{ if }x<y
\\
\nu_{y,x-y}\left(\eta_0\in\{1,2\}, \eta_1=1\right)&\text{ if }y<x
\end{cases},
\end{equation}
where $\nu_{\lambda_1, \lambda_2}$ is the 
ergodic translation-invariant stationary distribution 
(of a configuration $\eta\in\{1,2,\infty\}^\ZZ$)
for the $2$-type ASEP on $\ZZ$ 
with densities $\lambda_1$ and $\lambda_2$
of type-$1$ and type-$2$ particles respectively,
given by Theorem 
\ref{thm:2statline}.

Similarly, for any $x_0, x_1, \dots, x_{r-1}\in(0,1)$,
the probability of any event of the form
\begin{equation}\label{projectionN}
\{W^{(L)}_i\leq x_1, W^{(L)}_{i+1}\leq x_2,\dots, 
W^{(L)}_{i+r-1}\leq x_{r}\}
\end{equation}
converges as $L\to\infty$ to a limit that does not depend on
$i$ (which may be written in terms of a suitable 
$r$-type stationary distribution for the ASEP on $\ZZ$).
\end{proposition}

\begin{proof}
Suppose $0=u_0<u_1<\dots<u_{N}<1$, 
and a function $h:(0,1)\to\{1,2,\dots, N,\infty\}$
is defined by 
\begin{equation}\label{fprojection}
h(x)=
\begin{cases}
n &\text{ if } u_{n-1}<x\leq u_n, \text{ for } n=1,2,\dots, N;\\
\infty &\text{ if } x>u_N.
\end{cases}
\end{equation}
Then $\left(h(W^{(L)}_i),  
-\lfloor\frac L2\rfloor \leq i \leq \lceil\frac L2\rceil-1\right)$
has the stationary distribution of the $N$-type ASEP 
on the ring of size $L$, with $k_n$ particles of type $n$ 
for $n=1,\dots, N$, where $k_n=\lfloor u_n L \rfloor - \lfloor u_{n-1} L \rfloor$.

Putting $N=2$ and $u_1=\min\{x,y\}$, $u_2=\max\{x,y\}$,
we have $h(W_i^{(L)})=1$ iff $W_i^{(L)}\leq u_1$,
and $h(W_i^{(L)})\in\{1,2\}$ iff $W_i^{(L)}\leq u_2$. 
Then (\ref{projection2}) follows directly from 
Proposition \ref{prop:locallimit} with $\lambda_1=u_1$,
$\lambda_2=u_2-u_1$. 

More generally for (\ref{projectionN}), 
let $(u_1, \dots, u_N)$ be the increasing reordering of 
$(x_1, \dots, x_N)$. Then we can apply Proposition 
\ref{prop:locallimit} with $\lambda_n=u_n-u_{n-1}$, $n=1,\dots, N$,
to obtain the limit in (\ref{projectionN}) 
in terms of the $N$-type stationary distribution 
$\nu_{\lambda_1, \dots, \lambda_N}$.
Since that distribution is translation invariant on $\ZZ$,
the limit is the same for all $i$. 
\end{proof}

Convergence of all events of the form of (\ref{projectionN}) 
is enough to show that the sequence $W^{(L)}$ has a single
distributional limit point, and this gives the distribution of 
$W$ required for part (a) of Theorem \ref{thm:Wlimit}.
Since $W^{(L)}_i$ has uniform distribution on 
$\{1/L, 2/L, \dots, 1\}$, we have that $W_i\sim\text{Uniform}[0,1]$.
In fact, with little more work we could obtain
that the distribution of $W$ is the unique 
ergodic translation-invariant distribution for
the ASEP on $\ZZ$ whose marginals have Uniform$[0,1]$ distribution
(but we won't need to use this fact directly). 

For parts (b) and (c), we will analyse 
the distribution of $W$ via its projections onto 
$2$-type systems as in (\ref{projection2}),
using the queueing construction of the stationary
distribution $\nu_{\lambda_1, \lambda_2}$ of the $2$-type ASEP on $\Z$. 
Let us recall that construction.
We have an stationary queueing system in discrete time.
At each time $i\in\ZZ$, we have an arrival 
with probability $\lambda_1$, and independently 
a potential service with probability $\mu:=\lambda_1+\lambda_2$. 
When a potential service occurs, if an arrival has occurred
at the same time-step, then a departure occurs with probability 1;
if no arrival occurred, then a departure occurs with probability
$1-q^{k}$ where $k$ is the number of customers present in the queue,
and otherwise (with probability $q^k$) an unused service occurs. 
Then set $\eta_i=1$ if a departure occurs at $i$, $\eta_i=2$ 
if an unused service occurs at $i$, and $\eta_i=\infty$ if 
no service was available at $i$. At a given time $i$,
the marginal probability that $\eta_i=1$ is $\lambda$,
that $\eta_i=2$ is $\mu-\lambda$, and that $\eta=\infty$ is $1-\mu$.

First, for part (c), we use arguments
similar to those used for related calculations 
concerning the TASEP speed process in the $q=0$
case in \cite{AAV}, although the analysis is
more complicated now that $q>0$, 
since unused services can occur even when the queue is not empty. 

For convenience we write $\nu_{\lambda,\mu-\lambda}(a,b)$
for the probability that $\eta_0=a, \eta_1=b$ under 
$\nu_{\lambda, \mu-\lambda}$, where $a,b\in\{1,2,\infty\}$.

Recall that we write $f$ for the density of $(W_0,W_1)$ 
on $\{(x,y)\in[0,1]^2:x\ne y\}$, and $f^*$ for the density
along the diagonal $x=y$. 

\begin{lemma}$ $
\label{lemma:projections}
\begin{itemize}
\item[(i)]
$\nu_{\lambda,\mu-\lambda}(2,\infty)=(\mu-\lambda)(1-\mu)$;
\item[(ii)]
$\nu_{\lambda,\mu-\lambda}(2,2)=(\mu-\lambda)\big[(1-\lambda)\mu-q\lambda(1-\mu)\big]$;
\item[(iii)]
$\nu_{\lambda,\mu-\lambda}(2,1)=(\mu-\lambda)\big[\lambda\mu+q\lambda(1-\mu)\big]$;
\item[(iv)]
\[
f(x,y)=
\begin{cases}
\frac{d^2}{dxdy}\nu_{x,y-x}(2,\infty)& \text{ for }x<y,\\
\frac{d^2}{dxdy}\nu_{y,x-y}(2,1)& \text{ for }x>y,
\end{cases}
\,\,\,
f^*(x)=\lim_{\epsilon\downarrow0}
\frac{1}{\epsilon}\nu_{x,x+\epsilon}.
\]
\end{itemize}
\end{lemma}

Note in particular that as $\mu-\lambda\to0$, the density of second-class particles goes to 0, but conditional on seeing a second-class particle at a given site, 
the probability that the next site also contains a second-class particle stays bounded away from 0. 

Part (i) is straightforward. The second or third part are then easily seen to follow from each other, but are more complicated to establish. Note that when $q=0$, an unused service can only occur when the queue is empty,
and the conditional probability of the output of the next time-slot is
then easy to deduce. For $q>0$ on the other hand, an unused service occurs
with probability $q^n$ when the queue-length is $n$. As a result, 
conditioning on an unused service constitutes an exponential tilting
of the stationary distribution of the system.

\begin{proof}[Proof of Lemma \ref{lemma:projections}]
For part (i), we need the probability of observing an unused service at time 0, 
followed by no service at time 1. The probability of an unused service at time 0 
is $\mu-\lambda$, and a service occurs at time 1 with probability $\mu$ independently
of everything that has occurred earlier, giving $(\mu-\lambda)(1-\mu)$ overall as required. 

We turn to part (ii). Since the probability of an unused service at time 1 is $\mu-\lambda$,
we need to show that conditional on an unused service at time 1, the probability 
of another unused service at time 2 is 
\begin{equation}\label{needcond}
(1-\lambda)\mu-q\lambda(1-\mu).
\end{equation}

First let's consider the process of the length of the queue after each service, 
which is a Markov chain. It's a birth-and-death chain on $\ZZ_+$, and for all $k>0$, 
\begin{align*}
p_{k-1,k} &=\lambda(1-\mu),
\\
p_{k,k-1} &=(1-\lambda)\mu(1-q^k).
\end{align*}
Hence its stationary distribution $(\pi_k, k\geq 0)$ satisfies 
\begin{align}
\label{pirecursion}
\pi_k&=\frac{\lambda}{1-\lambda}\frac{1-\mu}\mu \frac1{1-q^k}\pi_{k-1}
\\
\nonumber
&=\left(\frac\lambda{1-\lambda}\frac{1-\mu}\mu\right)^k
\frac1{1-q}\frac1{1-q^2}\dots\frac1{1-q^k}\pi_0.
\end{align}
The probability of seeing a second-class particle at a given site is
$\sum_k\pi_k(1-\lambda)\mu q^k$, while the probability of seeing
two such particles at a given pair of successive sites is
$\sum_k\pi_k(1-\lambda)^2\mu^2 q^{2k}$.

So we wish to show that (\ref{needcond}) 
is equal to 
\[
\frac{\sum_k\pi_k(1-\lambda)^2\mu^2 q^{2k}}{\sum_k\pi_k(1-\lambda)\mu q^k}.
\]
We claim that for any $\alpha<1$ and $q<1$, the following identity holds:
\begin{equation}\label{aq}
\sum_{k=0}^\infty \frac{\alpha^k q^{2k}}{(1-q)\dots(1-q^k)}
=(1-\alpha q)\sum_{k=0}^\infty \frac{\alpha^k q^k}{(1-q)\dots(1-q^k)}.
\end{equation}
The desired equality then follows from (\ref{aq}) with 
$\alpha=\frac{\lambda}{1-\lambda}\frac{1-\mu}\mu$.

To obtain (\ref{aq}), we may write the difference between the RHS and the LHS as
\begin{multline*}
\sum_{k=0}^{\infty}
\frac{
\alpha^k q^k-\alpha^{k+1}q^{k+1}-\alpha^k q^{2k}
}
{
(1-q)\dots(1-q^k)
}
\\
\begin{split}
&=\sum_{k=0}^\infty
\frac{
\alpha^k q^k(1-q^k)
}
{
(1-q)\dots(1-q^k)
}
-
\sum_{k=0}^\infty
\frac{
\alpha^{k+1} q^{k+1}
}
{
(1-q)\dots(1-q^k)
}
\\
&=\sum_{k=1}^\infty
\frac{
\alpha^k q^k(1-q^k)
}
{
(1-q)\dots(1-q^k)
}
-
\sum_{k=0}^\infty
\frac{
\alpha^{k+1} q^{k+1}
}
{
(1-q)\dots(1-q^k)
}
\\
&=\sum_{k=1}^\infty
\frac{
\alpha^k q^k
}
{
(1-q)\dots(1-q^{k-1})
}
-
\sum_{m=1}^\infty
\frac{
\alpha^{m} q^{m}
}
{
(1-q)\dots(1-q^{m-1})
}
\\
&=0,
\end{split}
\end{multline*}
as required.

This establishes part (ii). Now part (iii) follows 
because the expressions in (i), (ii) and (iii)
must sum to the probability of observing a
second-class particle at site 1, which is $\mu-\lambda$.

Finally, for part (iv), the claimed expressions for the density
$f(x,y)$ are continuous away from $x=y$ and the expression for $f^*(x)$
is continuous in $x$. Then for $f(x,y)$, it suffices to check that for $y<x$,
\begin{align*}
\textstyle\frac{d^2}{dxdy}\PP(W_0\leq x, W_1\leq y)
&=\textstyle\frac{d^2}{dxdy}\PP(y< W_0\leq x, W_1\leq y)\\
&=\textstyle\frac{d^2}{dxdy}\nu_{y,x-y}(2,1),
\end{align*}
and similarly for $x<y$,
\begin{align*}
\textstyle\frac{d^2}{dxdy}\PP(W_0\leq x, W_1\leq y)
&=\textstyle\frac{d^2}{dxdy}\PP(x< W_0\leq y, W_1>y)\\
&=\textstyle\frac{d^2}{dxdy}\nu_{x,y-x}(2,\infty).
\end{align*}
Finally on the diagonal, for $f^*(x)$ we have
\[
\lim_{\epsilon\to0}\frac{1}{\epsilon}\PP(W_0\in[x,x+\epsilon),
W_1\in[x,x+\epsilon))=\lim_{\epsilon\to0}\frac{1}{\epsilon}
\nu_{x,\epsilon}(2,2)
\]
as required. 
\end{proof}

For part (b) of Theorem \ref{thm:Wlimit}, we obtain
(\ref{fxy}) and (\ref{fstarx}) by substituting the 
expressions in parts (i)-(iii) of Lemma \ref{lemma:projections}
into part (iv). Then (\ref{W1W2}) follows by
integrating over $x$ and $y$ in (\ref{fxy}), and over $x$
in (\ref{fstarx}).

Finally we turn to part (c) of Theorem \ref{thm:Wlimit},
which is the most involved. It follows from the following result:

\begin{proposition}
\label{prop:infinitelymany}
Define a random set $\cU\subset\{1,2,\dots\}$ in the following way. 

Let $X\sim\textrm{Uniform}[0,1]$.

Define a random walk $Q_i, i\geq 1$ as follows. 
Choose $Q_1\geq 0$ from the distribution $(\tpi_k, k\geq 0)$
satisfying 
\begin{equation}\label{xxinitial}
\tpi_k=\frac{q}{1-q^k}\tpi_{k-1}.
\end{equation}
Now for each $i=1,2,\dots$, given $X=x$ and $Q_{i}=k$,
\begin{alignat}{2}
\nonumber
&\text {with probability } x(1-x): &&\text{ let }Q_{i+1}=Q_{i}+1, \text{ and } i\notin \cU.
\\
\nonumber
&\text {with probability } x^2+(1-x)^2: &&\text{ let }Q_{i+1}=Q_{i}, \text{ and } i\notin \cU.
\\
\label{xxwalk}
&\text {with probability } x(1-x)q^k: &&\text{ let }Q_{i+1}=Q_{i}, \text{ and } i\in \cU.
\\
\nonumber
&\text {with probability } x(1-x)(1-q^k): &&\text{ let }Q_{i+1}=Q_{i}-1, \text{ and } i\notin \cU.
\end{alignat}
Then the set $\cU$ is infinite with probability $1$, 
and is stochastically dominated by the set 
$\{i\geq 1: W_i=W_0\}$. 
\end{proposition}

Here $Q$ plays the role of a queue-length process,
and $\cU$ plays the role of the set of unused service times,
in a queue whose arrival rate and service rate are both equal to $x$. 
(Such a queue is null recurrent.)

Although here we only claim that $\cU$ is a stochastic lower bound
for the set $\{i\geq 1: W_i=W_0$\}, it actually holds that the two
have the same distribution. In fact, we can be more precise: 
there is a system of regular conditional probabilities such 
that conditional on $W_0=x$, the distribution of $\{i\geq1:W_i=W_0\}$
is the distribution of $\cU$ obtained by the construction of Proposition 
\ref{prop:infinitelymany} given $X=x$. 
The argument to establish these stronger statements may be
apparent from the proof below, but we won't fill in the details.

\newcommand{\Qlm}{Q^{\lambda,\mu}}
\newcommand{\Alm}{\cU^{\lambda,\mu}}
\newcommand{\tpilm}{\tpi^{\lambda,\mu}}

The rest of this section is devoted to the proof of Proposition 
\ref{prop:infinitelymany}. First we compare the construction 
in Proposition \ref{prop:infinitelymany} to the queue-length
construction of the $2$-type stationary distribution on $\ZZ$:

\begin{lemma}\label{lemma:AlmAxxcomparison}
Fix any $\lambda, \mu$ with $0\leq \lambda < \mu \leq 1$. 

Define a random walk $\Qlm$ and a random set 
$\Alm\subset\{1,2,\dots\}$ as follows. 

Choose $\Qlm_1$ from the distribution $(\tpilm_r, r\geq 0)$ 
satisfying 
\begin{equation}\label{lminitial}
\tpilm_r=\frac{\lambda}{1-\lambda}\frac{1-\mu}{\mu}\frac{q}{1-q^r}\tpi_{r-1}.
\end{equation}
Now for each $i=1,2,\dots$, given $\Qlm_{i}=r$,
\begin{alignat}{2}
\nonumber
&\text {with probability } \lambda(1-\mu): &&\text{ let }\Qlm_{i+1}=\Qlm_{i}+1, \text{ and } i\notin \Alm.
\\
\nonumber
&\text {with probability } \lambda\mu+(1-\lambda)(1-\mu): &&\text{ let }\Qlm_{i+1}=\Qlm_{i}, \text{ and } i\notin \Alm.
\\
\label{lmwalk}
&\text {with probability } (1-\lambda)\mu q^r: &&\text{ let }\Qlm_{i+1}=\Qlm_{i}, \text{ and } i\in \Alm.
\\
\nonumber
&\text {with probability } (1-\lambda)\mu(1-q^r): &&\text{ let }\Qlm_{i+1}=\Qlm_{i}-1, \text{ and } i\notin \Alm.
\end{alignat}

\newcommand{\Axx}{\cU^{x,x}}

Write $\Axx$ for a set whose distribution is that of the set $\cU$ in Proposition \ref{prop:infinitelymany} conditioned on $X=x$. 
For $\lambda\leq x\leq \mu$, the set $\Alm$ 
stochastically dominates the set $\Axx$. 
\end{lemma}

\begin{proof}
We compare the initial distributions and transition probabilities for the
``$(\lambda,\mu)$-walk" at (\ref{lminitial})-(\ref{lmwalk}) and the ``$(x,x)$-walk" at (\ref{xxinitial})-(\ref{xxwalk}). 

Since $\lambda<\mu$, the distribution defined by (\ref{lminitial})
is dominated by the distribution defined by (\ref{xxinitial}). 
Furthermore, the ``up-step" probability $\lambda(1-\mu)$ in 
(\ref{lmwalk}) is smaller than the corresponding probability $x(1-x)$ 
in (\ref{xxwalk}), and the down-step probability 
$(1-\lambda)\mu(1-q^r)$ is larger than the corresponding 
probability $x(1-x)(1-q^k)$ when $r=k$. It follows that 
there is a coupling of the walks $(Q_i, \Qlm_i, i\geq 1)$
which always stays in the set $\{(k,r): k\geq r\}$. 

Further, the probability of including the next time-step $i$ 
in the set $\cU^{x,x}$ (or $\Alm$ respectively) is smaller
at (\ref{xxwalk}) than at (\ref{lmwalk}) whenever $k\geq r$,
since then $x(1-x)q^k<(1-\lambda)\mu q^r$. 
From this it's straightforward to extend the coupling
in the previous paragraph to a coupling of $(Q, \Qlm, \cU^{x,x}, \Alm)$
such that with probability $1$, $\cU^{x,x}\subseteq\Alm$, as required. 
\end{proof}

\begin{lemma}\label{lemma:Almqueue}
The distribution of the set $\Alm$ defined in Lemma
\ref{lemma:AlmAxxcomparison} 
is the same as the distribution of
the set $\{i>0: \lambda\leq W_i < \mu\}$ conditional on the event
$\{\lambda \leq W_0\leq \mu\}$. 
\end{lemma}

\begin{proof}
Consider the projection $f$ from (\ref{fprojection}),
with $u_1=\lambda$ and $u_2=\mu$. The configuration $h(W)
=(h(W_i), i\in\ZZ)$ has the 
$2$-type stationary
distribution $\nu_{\lambda,\mu-\lambda}$
In particular, the sites of type-$2$ particles in $f(W)$
are precisely those $i$ such that $\lambda\leq W_i<\mu$.

The walk (\ref{lmwalk}) is exactly the queue-length process 
with arrival rate $\lambda$ and service rate $\mu$
used to generate the stationary distribution $\nu_{\lambda, \mu-\lambda}$ of the $2$-type ASEP on $\ZZ$, and the adding of a point
to $\Alm$ corresponds to an unused service in the queue. 

Finally, the initial distribution $\tpilm$ of $\Qlm_1$ is 
the distribution of the queue-length in equilibrium 
conditioned on an unused service having occurred at the previous
time-step.
To verify this, note that such a distribution should be 
proportional to $q^k\pi_k$
where $\pi_k$, the unconditioned equilibrium queue-length 
distribution, satisfies (\ref{pirecursion}). This gives  
the recursion (\ref{lminitial}) as required.
\end{proof}

Combining the last two lemmas, we have that the distribution of the set 
$\{i>0: \lambda\leq W_i < \mu\}$  conditional on the event
$\{\lambda \leq W_0\leq \mu\}$
dominates the distribution of the set $\cU$ conditional on 
$\lambda\leq X<\mu$.

Fix any $N$, and let $m\in\{0,1,\dots, N-1\}$.
As just observed, conditional on 
$\{m/N \leq W_0\leq (m+1)/N\}$,
the set
$\{i>0: m/N\leq W_i < (m+1)/N\}$
dominates the distribution of $\cU$ given 
$\{m/N\leq X < (m+1)/N$.

But if $m/N\leq W_0\leq (m+1)/N$, then the 
set
$\{i>0: m/N\leq W_i < (m+1)/N\}$
is contained in the
set $\{i>0:|W_i-W_0|<1/N\}$.
So we get that conditional on 
$m/N\leq W_0\leq (m+1)/N$,
the set $\{i>0:|W_i-W_0|<1/N\}$
dominates the distribution of $\cU$
given $\{m/N\leq X < (m+1)/N$.

But both $W_0$ and $X$ have Uniform$[0,1]$ distribution,
and so each lies in any interval $[m/n, (m+1)/N]$ with 
probability $1/N$. 
So we may multiply by $1/N$ and sum over $m$,
to obtain that the set $\{i>0:|W_i-W_0|<1/N\}$ 
dominates the set $\cU$.

But also the sets $\{i>0:|W_i-W_0|<1/N\}$ form a decreasing
family as $N$ increases. If they dominate $\cU$ for every $N$,
then also their intersection must dominate $\cU$. 
But the intersection is exactly the set $\{i>0:W_i=W_0\}$. 

Finally note that $\cU$ is infinite with probability $1$,
since conditional on any value of $X=x\in (0,1)$, the 
walk given by (\ref{xxwalk}) is null recurrent,
and each time the walk hits $0$ the next site is added to
the set $\cU$ with probability $x(1-x)$ independently of 
all past choices. 

This completes the proof of Proposition \ref{prop:infinitelymany},
and hence of Theorem \ref{thm:Wlimit}.

\section{Alternative queueing construction}
\label{sec:alternative}
The queueing discipline used in the constructions in 
this paper has a somewhat unnatural feature, in that
a customer who has just arrived at the queue is treated differently
from one who was already present since the previous time-step.
A customer who has just arrived will always accept any
offered service, while a customer who was previously present
rejects any offer with probability $q$.

One can also apply equivalent constructions with this
distinction removed. Now each customer rejects
each service with probability $q$, irrespective
of time of arrival. We conjecture that this alternative
model gives exactly the same distribution, and in particular realises
the multi-type ASEP stationary distribution
(both on $\ZZ$ and on $\ZZ_L$). 

This has been computationally verified for values of $L$ up to 6. 
In the simplest case $N=2$, it is not hard to verify more generally
using a slightly different matrix product realisation. 
Rather than taking $X_1^{(2)}=I+\mepsilon$, $X_2^{(2)}=\malpha$,
$X_\infty^{(2)}=I+\mdelta$ where $\mepsilon$, $\malpha$
and $\mdelta$ have the form written at (\ref{adedef}),
one can take instead 
\begin{gather*}
\nonumber
X_2^{(2)}=A:=\begin{pmatrix} 
1 & q & 0 & 0 & \ldots \\
0 & q & q^2 & 0 &\ldots \\
0 & 0 & q^2 & q^3 &\ldots \\
0 & 0 & 0 & q^3 & \ldots \\
\vdots & \vdots & \vdots & \vdots & \ddots 
\end{pmatrix},
\,\,\, 
X_\infty^{(2)}=E:=\begin{pmatrix} 
1 & 1 & 0 & 0 & \ldots \\
0 & 1 & 1 & 0 & \ldots \\
0 & 0 & 1 & 1 & \ldots \\
0 & 0 & 0 & 1 & \ldots \\
\vdots & \vdots & \vdots & \vdots & \ddots 
\end{pmatrix},\\ 
X_1^{(2)}=D:=\begin{pmatrix} 
1-q & 0 & 0 & 0 &\ldots \\
1-q & 1-q^2 & 0 & 0 &\ldots \\
0 & 1-q^2 & 1-q^3 & 0 &\ldots \\
0 & 0 & 1-q^3 & 1-q^4 & \ldots \\
\vdots & \vdots & \vdots & \vdots & \ddots 
\end{pmatrix}.
\end{gather*}
One can then verify that these satisfy the following
quadratic algebra of \cite{DJLStasep}:
\begin{gather}
\nonumber
AD=qDA+(1-q)A\\
\label{newQuadraticRelations}
EA=qAE+(1-q)A\\
\nonumber
ED-qDE=(1-q)(E+D).
\end{gather}
(The relation between these identities and (\ref{fundamentalrelations})
is that if $\delta$, $\epsilon$, $\alpha$ satisfy (\ref{fundamentalrelations}),
then $D=I+\delta$, $E=I+\epsilon$ and $A=\alpha$ satisfy 
(\ref{newQuadraticRelations}).)

To extend this from $N=2$ to higher $N$ along the lines of 
the proof above, we would need a variant of the 
result of Theorem \ref{PEMtheorem} covering
a different recursive system from the one presented
at (\ref{adef}). The new system would have more 
non-zero terms (in particular, since we can now have
a departure of lower priority than arrival, 
we would also have non-zero terms $a_{m,n}^{(N)}$ for all $m<n\leq N$
as in the first line of (\ref{adef})). 

An alternative approach comes from arguments based on dynamic reversibility, involving dynamics defined on the set of multi-line diagrams, as done originally for the TASEP in 
\cite{FerMarmulti}. Such arguments may also make it possible to extend to $q>0$ the ``generalised" multi-line construction described for the TASEP in 
\cite{AAMPtasep}, which extend the construction of 
\cite{FerMarmulti}, restoring the symmetry between particles
and holes and establishing a richer ``Hasse diagram" structure
connecting systems with different numbers of particle types. 

\section*{Acknowledgments}
Many thanks to Omer Angel, Arvind Ayyer, and Svante Linusson for valuable conversations related to this work, and to Lauren Williams and Sylvie Corteel for information about their recent work \cite{CorteelMandelshtamWilliams}.
I'm very grateful to two referees whose comments and suggestions
have considerably improved the paper.
\bibliographystyle{usual2}
\bibliography{jbm}

\end{document}

%% file: jbmmacros.tex
\newcommand{\Ch}[2]{\ensuremath{\begin{pmatrix} #1 \\ #2 \end{pmatrix}}}
\newcommand{\ch}[2]{\ensuremath{\left( \begin{smallmatrix} #1 \\ #2
\end{smallmatrix} \right)}}

\newcommand{\PP}{\mathbb{P}}
\newcommand{\ZZ}{\mathbb{Z}}

\newcommand{\NN}{\mathbb{N}}

\newcommand{\Z}{\mathbb{Z}}

\newcommand{\cS}{\mathcal{S}}

\newcommand{\cA}{\mathcal{A}}

\newcommand{\cH}{\mathcal{H}}

\newcommand{\cR}{\mathcal{R}}

\newcommand{\cU}{\mathcal{U}}

\newcommand{\bQ}{{\mathbf{Q}}}

\newcommand{\Bin}{\textrm{Bin}}

\newcommand{\tW}{{\tilde{W}}}

\newcommand{\tw}{{\tilde{w}}}

\newcommand{\tpi}{{\tilde{\pi}}}


%% file: cyclic.pdf_t
\begin{picture}(0,0)%
\includegraphics{cyclic.pdf}%
\end{picture}%
\setlength{\unitlength}{3947sp}%
\begingroup\makeatletter\ifx\SetFigFont\undefined%
\gdef\SetFigFont#1#2#3#4#5{%
  \reset@font\fontsize{#1}{#2pt}%
  \fontfamily{#3}\fontseries{#4}\fontshape{#5}%
  \selectfont}%
\fi\endgroup%
\begin{picture}(12207,2526)(436,-7675)
\put(12451,-7561){\makebox(0,0)[lb]{\smash{{\SetFigFont{20}{24.0}{\rmdefault}{\mddefault}{\updefault}{\color[rgb]{0,0,0}$T$}%
}}}}
\put(451,-7561){\makebox(0,0)[lb]{\smash{{\SetFigFont{20}{24.0}{\rmdefault}{\mddefault}{\updefault}{\color[rgb]{0,0,0}$-T$}%
}}}}
\put(5251,-7561){\makebox(0,0)[lb]{\smash{{\SetFigFont{20}{24.0}{\rmdefault}{\mddefault}{\updefault}{\color[rgb]{0,0,0}$-m$}%
}}}}
\put(6826,-7561){\makebox(0,0)[lb]{\smash{{\SetFigFont{20}{24.0}{\rmdefault}{\mddefault}{\updefault}{\color[rgb]{0,0,0}$m$}%
}}}}
\end{picture}%